\RequirePackage{fix-cm}
\documentclass[preprint,12pt]{elsarticle}

\usepackage{comment}
\usepackage[title]{appendix}
\usepackage{graphicx, xcolor} 
\usepackage[hyphens]{url}

\usepackage{amsmath, amssymb, bbm}
\usepackage{amsthm}
\newtheorem{theorem}{Theorem}[section]
\newtheorem{lemma}[theorem]{Lemma}
\newtheorem{proposition}[theorem]{Proposition}
\newtheorem{corollary}[theorem]{Corollary}
\newtheorem{assumption}[theorem]{Assumption}

\theoremstyle{definition}
\newtheorem{remark}{Remark}[section]

\newtheoremstyle{named}{}{}{\itshape}{}{\bfseries}{.}{.5em}{\thmnote{#3}}
\theoremstyle{named}
\newtheorem*{namedtheorem}{Theorem}

\renewcommand{\H}{\mathcal{H}} 

\begin{document}

\begin{frontmatter}

\title{Counting walks by their last erased self-avoiding polygons using sieves\tnoteref{ACKNOW}}
\tnotetext[label1]{This work is supported by the French National Research Agency (ANR) throught the ANR-19-CE40-0006 project \textsc{Al}gebraic \textsc{Co}mbinatorics of \textsc{H}ikes \textsc{O}n \textsc{L}attices (\textsc{Alcohol}).}
\author[PLG]{Pierre-Louis Giscard}
\address[PLG]{Universit\'e du Littoral C\^{o}te d'Opale, EA2597-LMPA\\
Laboratoire de Math\'ematiques Pures et Appliqu\'ees Joseph Liouville,\\ Calais, France.
\ead[Web page]{http://www-lmpa.univ-littoral.fr/~plgiscard/}
\ead{mailto:giscard@univ-littoral.fr}}

\begin{abstract} 
Let $G$ be an infinite, vertex-transitive lattice with degree $\lambda$ and fix a vertex on it. Consider all cycles of length exactly $l$ from this vertex to itself on $G$. Erasing loops chronologically from these cycles, what is the fraction $F_p/\lambda^{\ell(p)}$ of cycles of length $l$ whose last erased loop is some chosen self-avoiding polygon $p$ of length $\ell(p)$, when $l\to\infty$ ?   
We use combinatorial sieves to prove an exact formula for $F_p/\lambda^{\ell(p)}$ that we evaluate explicitly. We further prove that for all self-avoiding polygons $p$, $F_p\in\mathbb{Q}[\chi]$ with $\chi$ an irrational number depending on the lattice, e.g. $\chi=1/\pi$ on the infinite square lattice. In stark contrast we current methods, we proceed via purely deterministic arguments relying on Viennot's theory of heaps of pieces seen as a semi-commutative extension of number theory. Our approach also sheds light on the origin of the difference between exponents stemming from loop-erased walk and self-avoiding polygon models, and suggests a natural route to bridge the gap between both.
\end{abstract} 


\end{frontmatter}

\section{Context}
\subsection{Self-avoiding objects and loop-erased walks}
The ``widely open problem of counting Self-Avoiding Walks (SAWs) and Self-Avoiding Polygons (SAPs) on lattices'' (quoting Flajolet \& Sedgewick) was first conceived from the study of polymer chemistry, but these objects arise in a wide range of physical and mathematical problems; e.g. as phase boundaries \cite{Smirnov2010} and in percolation clusters \cite{Werner2000,Lawler2000,beffara2013} or, as B. Mandelbrot observed \cite{Mandelbrot1982}, from the outer frontier of Brownian motion \cite{Lawler2001}.  Because in such models SAWs and SAPs are invariably realised through random processes, the problem of studying and counting them has so far only been attacked with tools from statistical physics and probability theory. Works along these directions have yielded deep insights into renormalisation and conformal mappings such as the  relation between self-avoiding curves and the Schramm-Loewner Evolution (SLE) \cite{Lawler2004,beffara2008}, and the result of H. Duminil-Copin and S. Smirnov who proved the value of the connective constant of the honeycomb lattice \cite{Copin2012}. 

One of the most active subfields of this research concerns loop-erased random walks (LERWs) introduced by G. Lawler \cite{Lawler1980}. Lawler's original motivation was to produce yet another model of random generation of self-avoiding objects. Indeed, removing loops from an ordinary random walk in chronological order--the so called \textit{loop-erasing} procedure--yields the self-avoiding `skeleton' of the walk. LERWs are now the object of much research on conformal loop ensembles and loop soups \cite{Lawler20042,Lawler2007,Lawler2018}. 
At the heart of this activity are conformally invariant loop measures, one of which \cite{Lawler2018} associates to a SAP $p$ the proportion of all cycles--also known as closed walks--whose last erased loop is the SAP $p$. This measure, studied within the framework of probability theory, remains very difficult to evaluate explicitly. Indeed, this requires a complicated mapping to Abelian sand-pile models followed by arduous calculations which were completed only for the shortest SAPs (length less than 6). But the straightforward combinatorial meaning of the measure as a proportion means that a purely deterministic counting method should be able to evaluate it as well. This is indeed so, using number theoretic sieves and heaps of cycles.

\subsection{Heaps of cycles and number theory}
\hspace{5mm}The theory of heaps of pieces \cite{cartier1969,Viennot1986,Krattenthaler2006}, which describes the combinatorics of piles of arbitrary pieces, has now found a remarkable number of applications throughout mathematics, from combinatorics to group theory. Studying specifically heaps of cycles on graphs, called \emph{hikes}, P.-L. Giscard and P. Rochet showed that these constitute a very natural semi-commutative extension of number theory \cite{SIAM2017}. This extension comprises all the fundamental objects of number theory (zeta functions, sieve methods) as well as most of the relations between these. In this context, the hikes themselves extend the natural integers, the walks (a.k.a. paths) extend the integers of the form $p^k$ with $p$ prime and SAWs and SAPs extend the primes. This last observation implies that in the semi-commutative framework of the hikes,  \emph{the extension of the prime number theorem will give the asymptotic number of SAPs on regular lattices}. What is even more remarkable here is that the set of prime factors of a walk as dictated by the extension of number theory precisely coincides with the loops erased from a walk in Lawler's procedure. The fundamental premise of this work is that the tools offered by number theory are also effective in the study of self-avoiding objects.

\subsection{Layout}
In this self-contained contribution (definitions in \S\ref{DefNota}), we use purely deterministic sieve techniques  for evaluating the asymptotic fraction of all closed walks whose last erased loop is any given SAP $p$ on any finite graph (\S\ref{FiniteG}) or infinite vertex-transitive lattice (\S\ref{InfiniteG}). We arrive at an exact, closed-form, easy to evaluate formula for this fraction, immediately presented in the next section. Within the framework of probability-theory, this fraction is a conformally-invariant loop-measure known and studied by Lawler \cite{Lawler2018}. The sieves also make it clear that the loop-measure comes with seemingly hitherto unnoticed error-terms (Appendix~\ref{ErrorTerms}), which explain combinatorially the difference between LERW and SAP exponents. A route for overcoming this final hurdle using 1990s work is suggested (\S\ref{Discussion}).

\section{Statement of the main result and illustrations}
\subsection{Main theorem}
The main result, which we prove in the remainder of this work is as follows: 

\begin{namedtheorem}[Infinite Sieve Theorem]\label{InfiniteSieve}
Let $G$ be an infinite vertex transitive graph of bounded degree $\lambda$ and let $\mathsf{A}_G$ be its adjacency matrix. Let $p$ be a self-avoiding polygon on $G$. Let $\{G^{\text{Tor}}_N\}$ be a the small tori sequence of vertex-transitive graphs on $N$ vertices converging to $G$ as $N\to\infty$ (see Appendix~\ref{GraphSeq}). 
Let $R(z)$ be the ordinary generating function of all closed walks on $G$ with fixed initial vertex.
Then, on $G_N^{Tor}$, the fraction of all \underline{hikes}, i.e. heaps of cycles, which are closed walks whose last erased loop is $p$ is given asymptotically for $N\gg 1$ by
$$
\frac{\alpha^N}{N}\frac{F_p}{\lambda^{\ell(p)}}.
$$
In this expression $\alpha:=\lim_{z\to1/\lambda^-}\exp\left(\int \frac{1}{z}\big(R(z)-1\big) dz\right)$, $\alpha\in]0,1[$, is well defined and $F_p\lambda^{-\ell(p)}$ designates the fraction of all \underline{closed walks} defined up to translation on $G$ whose last erased loop is $p$. This fraction is explicitly given by 
\begin{equation}\label{FractionWalk}
\frac{F_p}{\lambda^{\ell(p)}} = \frac{1}{\lambda^{\ell(p)+1}} \,\,\mathsf{deg}^{\mathrm{T}}.\,\mathrm{adj}\left(\mathsf{I}+\mathsf{C}_G\big|_p.\mathsf{B}_p\right)\,.\,\mathsf{1},
\end{equation}
where $\mathrm{adj}(.)$ designates the adjugate operator, $\mathsf{B}_p$ is the adjacency matrix of the graph $G_p$ induced by $p$ and its immediate neighbours on $G$, $\mathsf{1}$ designates the vector full of $1$ and $\mathsf{deg}=\mathrm{diag}(\mathsf{B}_p^2)$ is the vector of vertex-degrees on $G_p$. Finally, $\mathsf{C}_G|_p$ is the restriction to $G_p$ of the matrix $\mathsf{C}_G:=\lim_{z\to1/\lambda^-}(\mathsf{I}-\mathsf{P}_\lambda)\mathsf{R}(z)$, with $\mathsf{P}_\lambda$ the projector onto the eigenspace associated with the dominant eigenvalue and $\mathsf{R}(z):=(\mathsf{I}-z\mathsf{A}_G)^{-1}$ is the resolvent of graph $G$. 
%
\end{namedtheorem}

\begin{corollary}
Let $G$ be an infinite vertex-transitive lattice of degree $\lambda$. Let $p$ be a self-avoiding polygon on $G$ and let $F_p\,\lambda^{-\ell(p)}$ be the fraction of all closed walks whose last erased loop is the SAP $p$. 
\begin{itemize}
\item[-]If $G$ is a $d$-dimensional hypercubic lattice, then $F_p\in\mathbb{Q}[1/\pi^{d-1}]$.
\item[-]If $G$ is the triangular, hexagonal or Kagom\'e lattice, then $F_p\in\mathbb{Q}[\sqrt{3}/\pi]$.
\end{itemize}
\end{corollary}

The error terms generated by the sieve on infinite graphs are given in Appendix~\ref{ErrorTerms} together with a discussion of their relevance for analytical estimates of the asymptotic growth of the number of SAPs of length $\ell$ as $\ell\to\infty$. See also the discussion of \S\ref{Discussion}.\\ 

From a practical point of view, we observe that the matrix $\mathsf{B}_p$ is of size $e(p)\times e(p)$, where $e(p)$ is the number of vertices of $G$ at distance at most 1 from $p$. Since clearly $\ell(p)\leq e(p)\leq \lambda \ell(p)$, computing the fraction $F_p/\lambda^{\ell(p)}$ costs $O\big(\ell(p)^3\big)$ operations.\\ 

The lattice constant $\alpha$ relates the \emph{densities} of walks and of hikes on the infinite lattice $G$. On the square lattice, $\alpha=\frac{1}{4} e^{\frac{4 C}{\pi }}\simeq 0.8025...$ with $C$ Catalan's constant.\\ 

The matrix $\mathsf{C}_G$ is easy to obtain on regular graphs because its entries obey the same recursion relations as the graph resolvent. More precisely, let $\big(\mathsf{C}_G\big)_{m,n}$ designate the entry of the matrix corresponding to jumping from vertex $m$ to vertex $n$. Then $\lambda\big(\mathsf{C}_G\big)_{m,n} = \sum_{i\in\mathcal{N}(n)}\big(\mathsf{C}_G\big)_{m,i}+\sum_{j\in\mathcal{N}(m)}\big(\mathsf{C}_G\big)_{j,n}$, where $\mathcal{N}(n)$ and $\mathcal{N}(m)$ designate the set of vertices that are neighbours to $n$ and $m$ on $G$, respectively. 
On the square lattice this implies that $\mathsf{C}_G$ has the following explicit expression:
$$
\big(\mathsf{C}_G\big)_{ij}=-\frac{1}{\pi}\int_{0}^{\infty}\frac{1}{\tau}\left(1-\left(\frac{\tau-\mathbbm{i}}{\tau+\mathbbm{i}}\right)^{x_{ij}-y_{ij}}\left(\frac{\tau-1}{\tau+1}\right)^{x_{ij}+y_{ij}}\right)d\tau
$$
where $\mathbbm{i}^2=-1$, $x_{ij}$ and $y_{ij}$ are the distance along $x$ and $y$ between vertices $i$ and $j$ of $G_p$, respectively. In particular if $x_{ij}=y_{ij}=m$ then 
$$
\big(\mathsf{C}_G\big)_{ij}=-\frac{4}{\pi}\sum_{k=1}^{m-1}\frac{1}{2k+1}=-\frac{2}{\pi}\left(H_{m-\frac{1}{2}}+\log (4)\right),
$$ 
with $H_m$ the $m$th harmonic number.
Explicit expressions for $\mathsf{C}_G$ have already been determined on many more lattices owing to its relation with lattice Green's functions and the resistor problem \cite{Atkinson1999,Cserti2011}.

\subsection{Illustrations}
\begin{figure}[!t]
\begin{center}
\vspace{-3mm}
\includegraphics[width=1\textwidth]{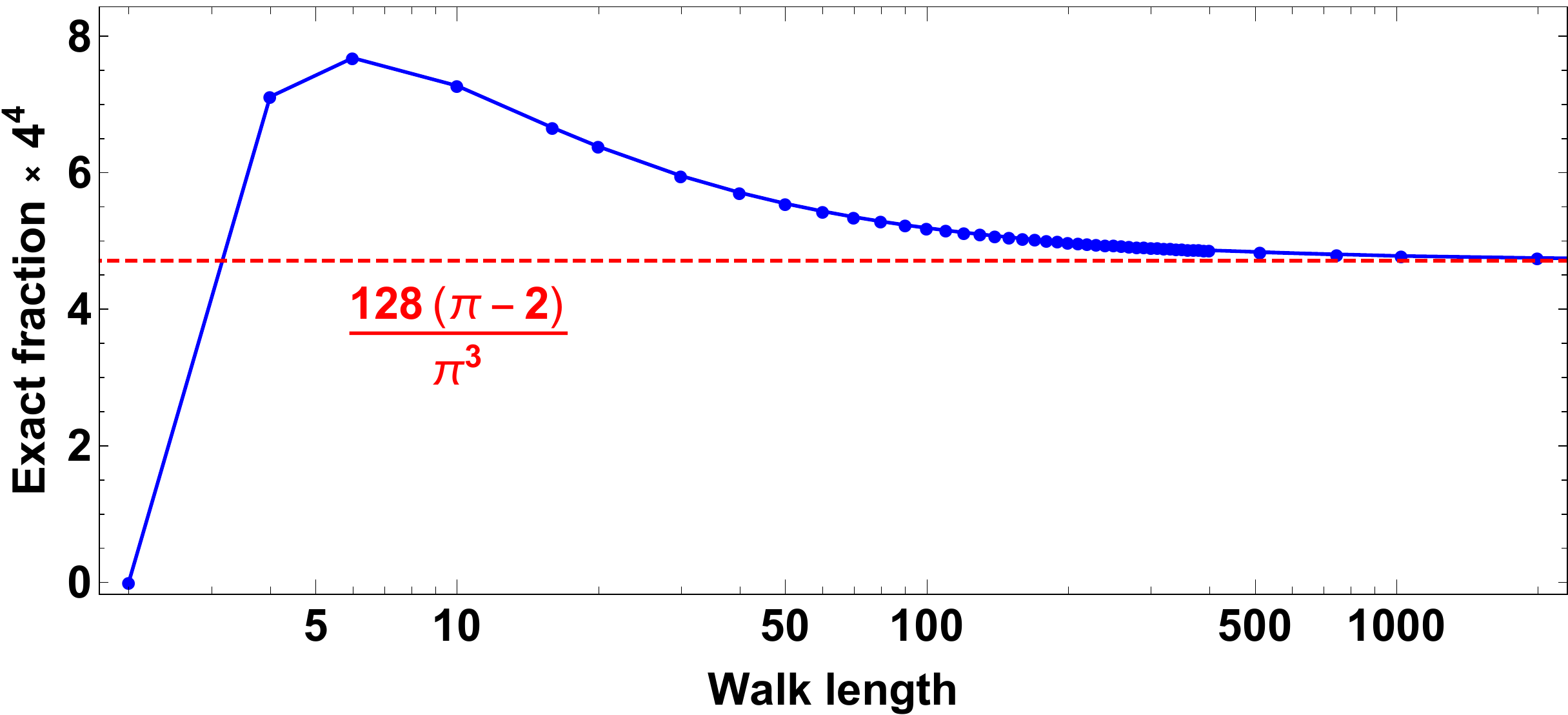}
\caption{Exact fraction of closed walks of length $\ell$ on the infinite square lattice whose last erased loop is a $1\times 1$ square as a function of the length of these walks. The exact fraction was obtained from the extension of Viennot's lemma to infinite graphs Proposition~\ref{InfiniteViennot} and Eq.~(\ref{R11}). The exact fraction converges to its asymptotic value proportionally with the inverse of the walk length, as dictated by an analysis of the error terms associated with the Infinite Sieve Theorem, see Appendix~\ref{ErrorTerms}.}
\label{fig:ConvergenceSquare}
\end{center}
\end{figure}
~\\[-1em]
We may now illustrate the Infinite Sieve Theorem with concrete results on the infinite square lattice.\\

\noindent $\blacktriangleright$ The fraction of closed walks whose last erased loop is a given edge $e$ is 
$$
\frac{F_e}{4^2} = \frac{1}{8}=0.125.$$
Since a point is connected to 4 edges, this means that an edge is the last erased loop of $1/2$ of all closed walks on the square lattice. In the language of the semi-commutative extension of number theory that holds for walks and hikes on graph, this indicates that $1/2$ of all closed walks on the square lattice whose unique right prime divisor is an edge cycle.\\ 

\noindent $\blacktriangleright$ The fraction of closed walks whose last erased loop is a $1\times1$ square is
\begin{equation}\label{FractionSquare}
\frac{F_{1\times1}}{4^4}=\frac{128 (\pi-2)}{4^4\pi^3}\simeq 0.0184.
\end{equation}
This result was first obtained via a complicated mapping between $F_p$ and Abelian sand-pile models \cite{Majumdar1991, Manna1992}. According to the authors this process ``becomes very tedious'' for longer self-avoiding polygons, so that very few explicit values for $F_p$ have been published in the literature so far, and all pertain to self-avoiding polygons of length $\ell\leq 6$.   

See Fig.~(\ref{fig:ConvergenceSquare}) for an illustration of the convergence of the fraction of closed walks on the infinite square lattice whose last erased loop is a $1\times 1$ square to the above number. Here the extension of Viennot's lemma to infinite graphs yields the ordinary generating function of closed walks whose last erased loop is an oriented $1\times 1$ square, which we denote $R_{1\times1}(z)$, as
\begin{align}\label{R11}
R_{1\times1}(z)&=\frac{1}{256\, \pi ^4
   z^4}\Big(\left(16 z^2-1\right)\, K(16 z^2)+E(16 z^2)\Big)^{\!2}\times\\
&\hspace{-10mm}   \Big(\left(1-16 z^2\right) K(16 z^2)^2+2 K(16 z^2) \left(8 \pi 
   z^2-E(16 z^2)\right)-4 \pi ^2 z^2+E(16 z^2)^2\Big),\nonumber\\
   &=z^4+12\, z^6+144\, z^8+1804\, z^{10}+23464\, z^{12}+\cdots,\nonumber
\end{align}
\begin{figure}
\begin{center}
\includegraphics[width=1\textwidth]{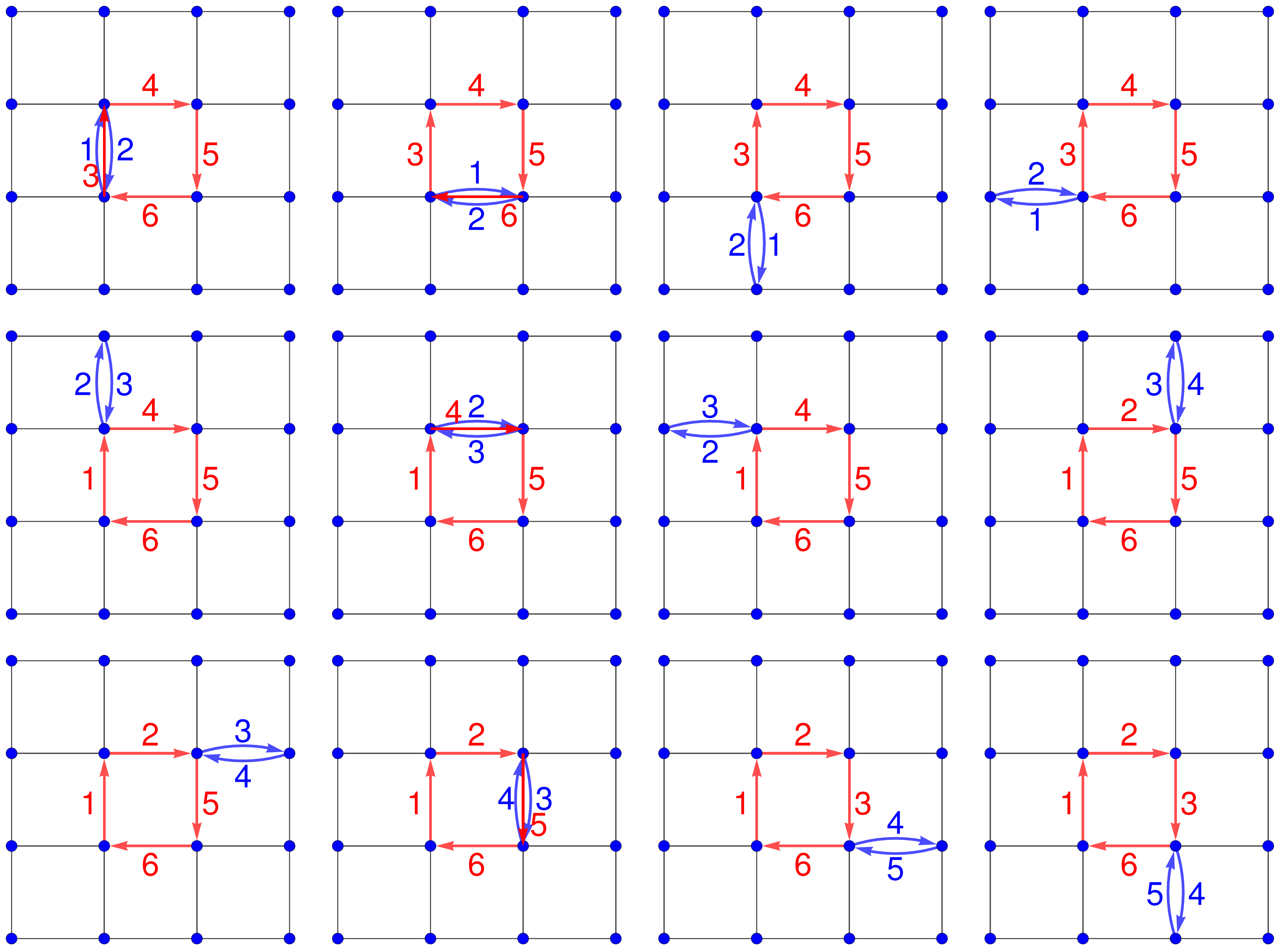}
  \end{center}
  \vspace{-7.5mm}
 \caption{\label{SquareMultiples}{\footnotesize Illustration of the 12 walks of length 6 on the square lattice whose last erased loop is right-oriented $1\times 1$ square cycle (highlighted in red). The numbers next to the highlighted edges indicate the order in which these edges are traversed.}}
  \vspace{-2mm}
\end{figure}
where $K(x):=\int_0^{\pi/2}\big(1-x\sin^2(\theta)\big)^{\!-1/2}d\theta$ and $E(x):=\int_0^{\pi/2}\big(1-x\sin^2(\theta)\big)^{\!1/2}d\theta$ are the complete elliptic integrals of the first and second kind, respectively. In Fig.~(\ref{SquareMultiples}) we illustrate the 12 walks of length 6 on the square lattice whose last erased loop is a $1\times 1$ oriented square, as correctly counted by the coefficient of $z^6$ in $R_{1\times 1}(z)$, denoted $[z^6]R_{1\times 1}(z)=12$.
Eq.~(\ref{FractionSquare}) establishes that asymptotically 
$$
[z^{2n}]R_{1\times1}(z)\sim\frac{128 (\pi-2)}{4^4\pi^3}\binom{2n}{n}^{\!\!2},\text{ as }n\to\infty.
$$

\noindent $\blacktriangleright$ The fraction of closed walks whose last erased loop is a $1\times2$ rectangle is
$$
\frac{F_{1\times2}}{4^6}=\frac{32 (\pi-8)(\pi-4)(3\pi-8)(3\pi-4)}{4^6\pi^4}\simeq 0.002585.
$$
$\blacktriangleright$ The fraction of closed walks whose last erased loop is a $1\times3$ rectangle is
\begin{align*}
\frac{F_{1\times 3}}{4^8}&=\frac{1024 (16-3 \pi) (64+3 (\pi -12) \pi ) (64+27 (\pi -4) \pi ) (128+\pi  (27 \pi -124))}{4^8\times 81 \pi ^7},\\
&\simeq 0.00035499.
\end{align*}
$\blacktriangleright$ The fraction of closed walks whose last erased loop is a $2\times2$ square is
$$
\frac{F_{2\times 2}}{4^8}=\frac{32768 (\pi -8)^2 (\pi -4) (3 \pi -8)^3 (9 \pi -32)}{4^8\times 81 \pi ^7}\simeq 0.00044623.
$$
~\\[-1em]

\noindent $\blacktriangleright$ As an example of longer SAP, consider:
\begin{figure}[h!]
\begin{center}
\vspace{-2mm}
\includegraphics[width=.2\textwidth]{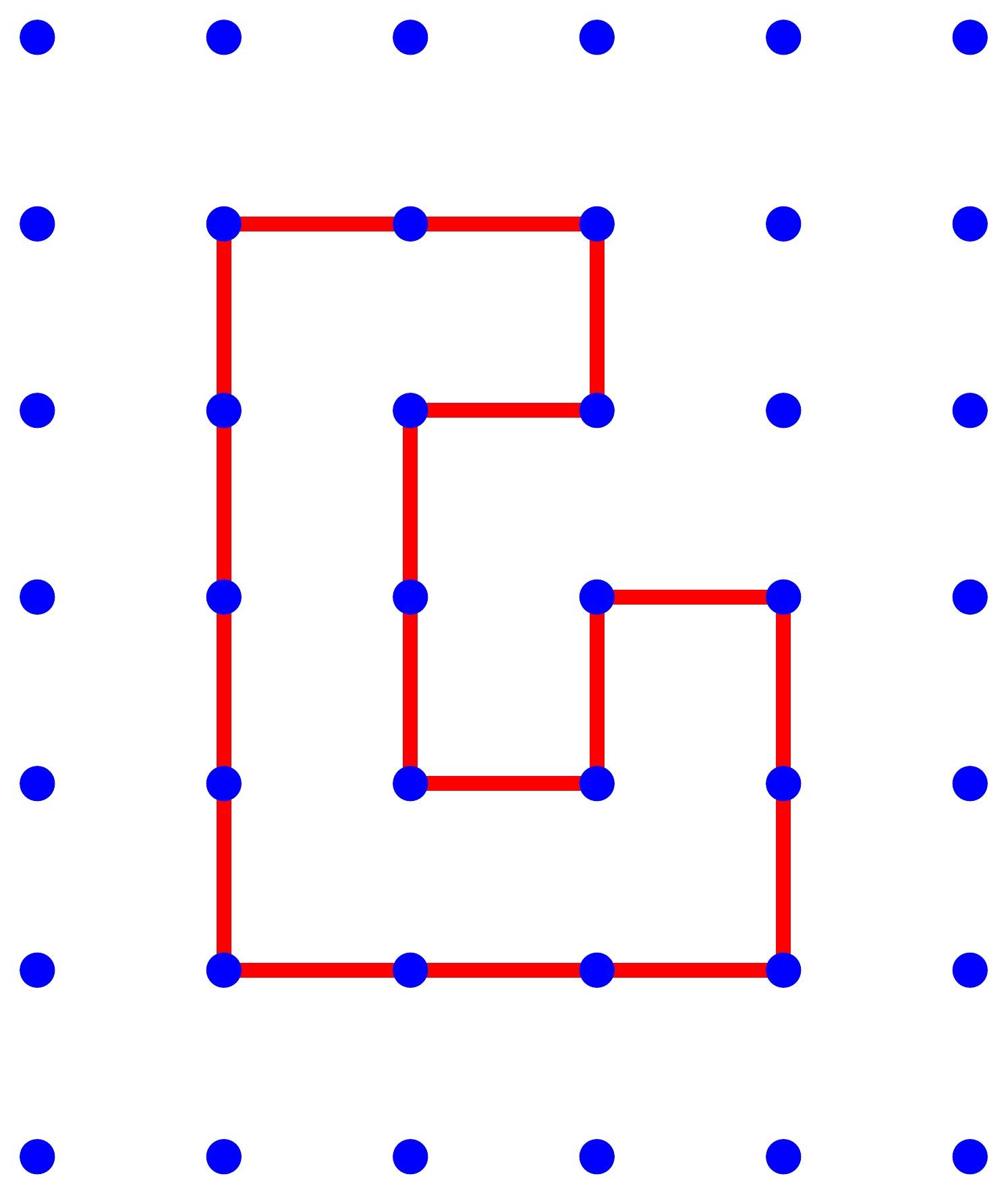}
\end{center}
\vspace{-6mm}
\end{figure}

\noindent Then the fraction of all closed walks whose last erased loop is this SAP is
\begin{align*}
\frac{F_{p}}{4^{18}}=&\frac{8388608}{4^{18}\times8303765625\, \pi ^{12}}\Big(1721510367131231944781594624\\
&-6733029120634416611029155840 \pi +12001725045126647537146527744 \pi ^2\\
&-12895675745638007921939841024
   \pi ^3+9303982639359984674575220736 \pi ^4\\
   &-4748903115679537036020154368 \pi
   ^5+1758418560456019196044640256 \pi ^6\\
   &-475910723284488375970037760 \pi
   ^7+93430267561362281294131200 \pi ^8\\
   &-12973459941155225172708000 \pi ^9+1209211981439562793530000
   \pi ^{10}\\
   &-67906363349663583525000 \pi ^{11}+1736896666805181140625 \pi ^{12}\Big),\\
   &\hspace{-4.4mm}\simeq 7.7644\times 10^{-9}.\\[-.7em]
\end{align*}
This result is well beyond what is realistically achievable from conformally invariant measures mapped to Abelian sandpile models. 

Fractions of walks whose last erased loop is a much longer SAP are easily obtained numerically, costing no more than $O(\ell(p)^3)$ to compute as outlined earlier. For example, the fraction of closed walks whose last erased loop in a $70\times 70$ square is
\begin{equation}\label{70Square}
\frac{F_{70\times 70}}{4^{280}}\simeq 1.5236\times10^{-108}.
\end{equation}
Analytically speaking these fractions become very involved very quickly as a function of SAP length and there is no reason to believe that there exists a simpler expression for them than that given by Eq.~(\ref{FractionWalk}) of the Infinite Sieve Theorem.  
For example, the analytical expression for the fraction of closed walks whose last erased loop is a $6\times 6$ square already involves sums and products of up to 16-digits prime integers. In fully expanded form this fraction involves a 67-digits prime integer (!). Overall, we have calculated the fractions for over 100,700,000 SAPs \emph{analytically}, and for more than 3,480,000,000 SAPs numerically on the square lattice. These results as well as the accompanying algorithm will soon be presented in a separate contribution. The hexagonal and triangular lattices will be also be treated.

\section{Definitions}\label{DefNota}
\subsection{Hikes and related objects}\label{Notation1}
In the general setting, we consider (weighted di)graphs $G = (\mathcal{V} ;\mathcal{E})$ with $N=|\mathcal{V}|$ nodes and $M=|\mathcal{E}|$ edges, both of which may be infinite but the degree of $G$ must be bounded. The ordinary adjacency matrix of $G$ is denoted $\mathsf{A}_G$ or simply $\mathsf{A}$. If $G$ is weighted then the entry $\mathsf{A}_{ij}$ is the weight of the edge $e_{ij}$ from $i$ to $j$ if this edge exists, and 0 otherwise. The labelled adjacency matrix of $G$ is denoted $\mathsf{W}$ and its entries are formal variables belonging to the Cartier-Foata monoid, $\mathsf{W}_{ij}=e_{ij}$.\\[-.85em]

A \textit{induced subgraph} $H$ of $G$, also called simply a \textit{subgraph} of $G$ and denoted $H\prec G$, is a set of vertices $ \mathcal{V}_H\subseteq \mathcal{V}$ together with the set of all edges linking these vertices in $G$, $\mathcal{E}_H=\{e_{ij}\in\mathcal{E}:\,i,j\in\mathcal{V}_H\}$.\\[-.85em]  

A \textit{walk} $w$ of length $\ell(w)$ from $v_i$ to $v_j$ on $G$ is a sequence $w = e_{i i_1} e_{i_1 i_2} \cdots e_{i_{\ell-1} j}$ of $\ell$ contiguous edges. The walk $w$ is \textit{open} if $i \neq j$ and \textit{closed} otherwise.\\[-.85em]

A \textit{simple cycle}, also known in the literature under the names \textit{loop}, \textit{cycle}, \textit{elementary circuit} and \textit{Self-Avoiding Polygon} or \textit{SAP}, is a closed walk $w = e_{i i_1} e_{i_1 i_2} \cdots e_{i_{\ell-1} i}$ which does not cross the same vertex twice, that is, the indices $i,i_1,\hdots,i_{\ell-1}$ are all different. Two simple cycles differing only by orientation are considered distinct (Rule 1), but two simple cycles differing only by their starting point are taken to be identical (Rule 2). The necessity of these choices and of the definition of hikes below can be found in Cartier and Foata's foundational work \cite{cartier1969}. From now on, we employ the letter $\mathcal{P}$ to designate sets of simple cycles, in particular $\mathcal{P}_G$ will be the set of all simple cycles on a graph $G$.\\[-.85em]

\begin{figure}
  \vspace{-20mm}
\begin{center}
\includegraphics[width=1\textwidth]{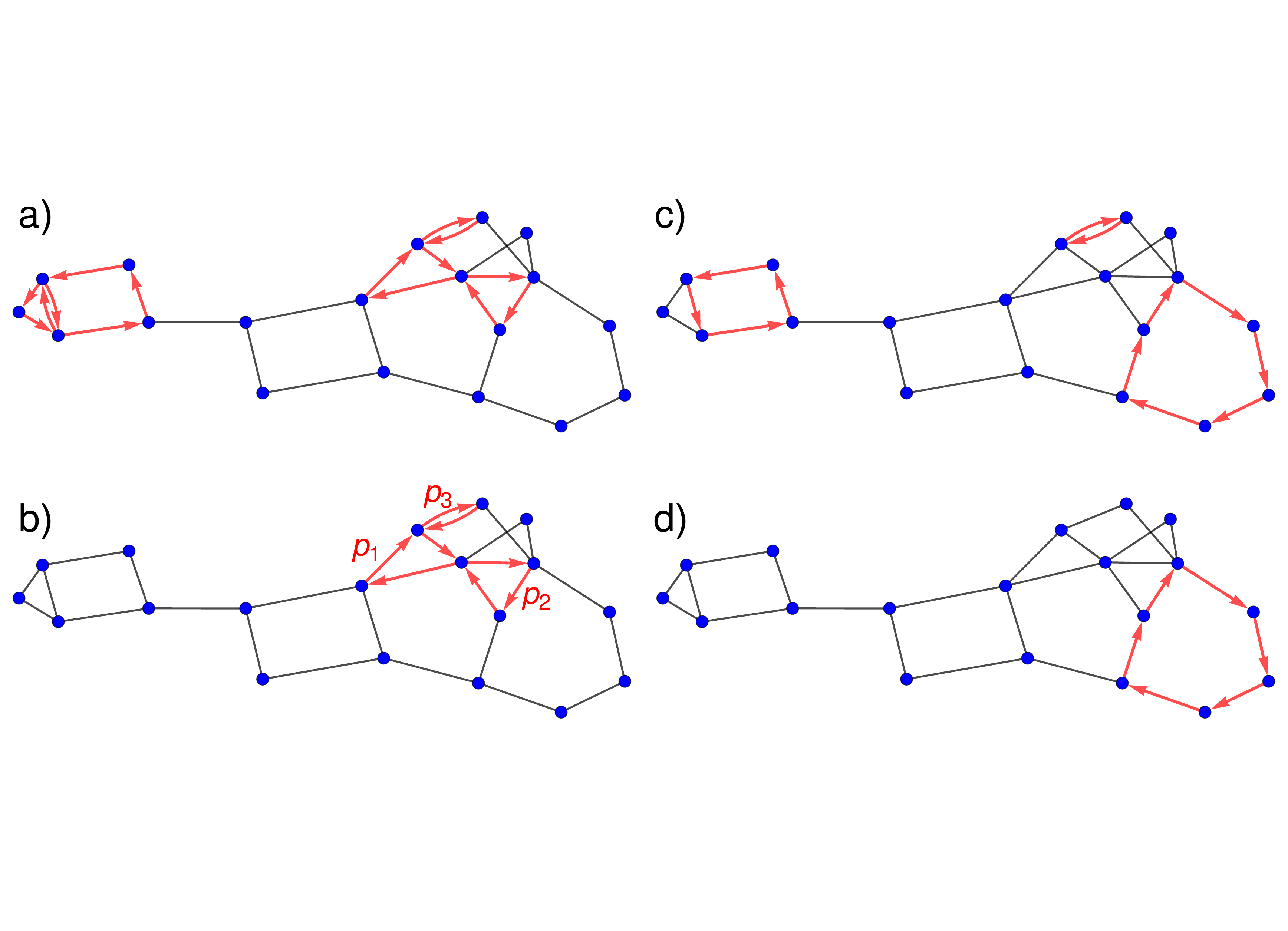}
  \end{center}
  \vspace{-25mm}
 \caption{\label{HikeIllustr}{\footnotesize Illustration of the main objects of the hike family on a graph $G$: \textbf{a)} a hike comprising two walks; \textbf{b)} a walk $w=p_2p_3p_1\equiv p_3p_2p_1$; \textbf{c)} a self-avoiding hike comprising three vertex disjoint simple cycles; \textbf{d)} a simple cycle. The walk $w$ has a unique right prime divisor, here the triangle $p_1$. In the Cartier-Foata monoids, the starting point of this walk can be any of the three vertices of $p_1$ and $p_1$ will always be the last erased loop of the walk.}}
  \vspace{-2mm}
\end{figure}
The central objects of the present work are \emph{hikes}, a hike $h$ being an unordered collection of vertex-disjoint closed walks. Hikes can be also be seen as equivalence classes on words $W=p_{i_1}p_{i_2}\cdots p_{i_n}$ over the alphabet of simple cycles $p_i$ of a graph. Two words $W$ and $W'$ are equivalent if and only if $W'$ can be obtained from $W$ through allowed permutations of consecutive simple cycles. In this context, two simple cycles are allowed to commute if and only if they are vertex disjoint $\mathcal{V}(p_i)\cap \mathcal{V}(p_j)=\emptyset \iff p_ip_j=p_jp_i$.\\[-.5em] 

For example, if $p_1$ and $p_2$ commute but neither commute with $p_3$, then $p_1p_2$ and $p_2p_1$ represent the same hike, but $p_1p_3p_2$ and $p_2p_3p_1$ are distinct hikes.\\[-.5em]

The letters $p_{i_1},\cdots, p_{i_n}$ found in a hike $h$ are called its \emph{prime} divisors. This terminology is due to the observation that simple cycles obey the defining property of prime elements in the semi-commutative Cartier-Foata monoid $\mathcal{H}$ of hikes equipped with the concatenation. In the context of Viennot's theory of heaps of pieces, hikes are heaps of simple cycles modulo Rules 1 and 2. In this work, we use the terminologies ``primes" and ``SAPs" interchangeably.\\[-.7em]
 
\noindent Two special types of hikes will be important for our purpose here:\\[-1.2em]

A \emph{self-avoiding hike} is a hike all prime factors of which commute with one another. In other terms, it is collection of vertex-disjoint simple cycles. If $\mathcal{P}$ designates a set of simple cycles, we designate $\mathcal{P}^{s.a.}$ the set of all \underline{s}elf-\underline{a}voiding hikes that can be built from cycles of $\mathcal{P}$. 

In Viennot's terminology, a hike is a \emph{walk} if and only if it is a pyramid of cycles, i.e. a heap of cycle with a unique top cycle. Equivalently, this means that a hike is a \emph{walk} if and only if it has a unique right prime divisor $p$ \cite{SIAM2017}. In this situation we say that $p$ right divides the walk $w$, denoted $p|_r w$ or that $w$ is a (left) multiple of $p$. Remarkably, these notions are also identical with those produced by G. Lawler's loop erasing procedure \cite{Lawler1980}: in this framework the unique right prime divisor of a closed walk is the last erased loop of this walk. As an example, let $p_1$, $p_2$ and $p_3$ be three simples cycles such that $p_2$ and $p_3$ commute with one another, but $p_1$ commutes with neither $p_2$ nor $p_3$. Then the hike $h=p_2p_3p_1\equiv p_3p_2p_1$ is a walk whose unique right prime divisor is $p_1$, i.e. $p_1$ is the last erased loop of $h$. This is illustrated in Fig.~(\ref{HikeIllustr}).
\\[-.85em]

\subsection{Functions on hikes}
A function on hikes is a complex-valued function $f:\mathcal{H}\mapsto \mathbb{C}$. The most important example here will be that of \emph{rank function}. 

A rank function, is a function $\rho:\mathcal{H}\mapsto \mathbb{R}$ that is totally additive over the hikes, $\rho(hh')=\rho(h)+\rho(h')$ and which respects the divisibility order, i.e. $h\leq h'\Rightarrow \rho(h)\leq \rho(h')$. The reverse implication does not hold in general. Examples of rank function include the length and the number of self-crossings of a hike. In general, we will denote $\rho(h)$ the rank of $h$ as per the rank function $\rho(.)$ and will denote $\varrho$ any given rank, i.e. value taken by the function $\rho$. In the case of the length function, $\ell(h)$ is the length of hike and $l$ denotes a length. The set of hikes with a given rank $\varrho$ is $\mathcal{H}_\varrho:=\{h\in\mathcal{H},~\rho(h)=\varrho\}$. The cardinality of this set is denoted $|\mathcal{H}_\varrho|:=\text{card}(\mathcal{H}_\varrho)$, if the graph is weighted this is understood to mean the total weight carried by hikes of rank $\varrho$.\\[-.85em] 

A function $f:\mathcal{H}\mapsto \mathbb{C}$ on hikes is associated with a formal series $\sum_{h\in\mathcal{H}}f(h) h$. This series is rarely accessible as such, rather linear algebra provides tools to access a related ordinary generating function $F(z):=\sum_{h\in\mathcal{H}}f(h) z^{\ell(h)}$. Important examples of functions on hikes and their related series include:\\[-.85em]  

The \emph{zeta function} on hikes is the identity function over the hikes $\forall h\in\mathcal{H}, \zeta(h)=1$. The associated ordinary generating function will  be denoted $\zeta(z)$, it is given by
$$
\zeta(z) = \sum_{h\in\mathcal{H}}z^{\ell(h)} = \frac{1}{\det\big(\mathsf{I}-z\mathsf{A}_G\big)},
$$
where $\mathsf{A}_G$ is the adjacency matrix of the graph $G$ on which the hikes live.\\[-.85em]  

The \emph{M\"{o}bius function} on hikes is the convolution inverse of the zeta function. We have 
$$
\forall h\in\mathcal{H},~ \mu(h)=\begin{cases}(-1)^{\Omega(h)},&\text{if $h$ is self-avoiding}\\0,&\text{otherwise}\end{cases}
$$
Here $\Omega(h)$ is the \emph{prime factor counting function}, its count the number of prime right-divisors of $h$. The associated ordinary generating function will be denoted $\mu(z)$, it is given by $\mu(z)=\det\big(\mathsf{I}-z\mathsf{A}_G\big)$.\\[-.7em] 
 
The \emph{walk von Mangoldt} function on hikes $\Lambda$ is defined as the number $\Lambda(h)$ of contiguous is the number of possible contiguous rearrangements of the edges in $h$, obtained without permuting two edges with the same starting point. This is equivalent to
$$
\forall h\in\mathcal{H},~\Lambda(h)=\begin{cases}
\ell(p),&\text{if $h$ is a walk with unique right prime divisor $p$},\\
0,&\text{otherwise}.
\end{cases}
$$
On a graph with $N$ vertices, the ordinary generating function associated with the walk von Mangoldt function is given by $\Lambda(z) = \mathrm{Tr}\big(\mathsf{I}-z\mathsf{A}_G\big)-N$.

\section{The asymptotics of hikes and walks on finite graphs}\label{FiniteG}
The aim of this section is to develop sieving tools to asymptotically count hikes satisfying certain properties on finite graphs. The main results here will be the Finite Sieve Theorem and its length corollary. Before we state and prove these results, there is an important precedent to be found in Viennot's work \cite{Viennot1986,Viennot1989}, which provides the ordinary generating functions of hikes which are closed walk multiples of any chosen prime $p$, i.e. whose last erased loop is the SAP $p$. Since the asymptotic expansion of this result is among the results obtained below, we start by recalling Viennot's result.

\begin{namedtheorem}[Viennot's lemma \cite{Viennot1986,Viennot1989}]
Let $G$ be a finite graph. Let $p$ be a prime on this graph and let $W_p:=\sum_{w:\,p|_rw} w$ be the formal series of closed walks whose unique right prime divisor is $p$. Then
$$
W_p=\frac{\det\big(\mathsf{I}-\mathsf{W}_{G\backslash p}\big)}{\det\big(\mathsf{I}-\mathsf{W}_G\big)}p.
$$
where $\mathsf{W}_{G\backslash p}$ and $\mathsf{W}_G$ designate the labelled adjacency matrices of $G\backslash p$ and $G$, respectively. 
\end{namedtheorem} 
Viennot gave a beautiful bijective proof of this result in \cite{Viennot1986}. There are at least four more proofs, one of which is the spirit of sieves and is provided below. The extension of Viennot's lemma to infinite graphs is obtained in Section~\ref{ViennotInfinite}. 
\begin{proof}
Let $\mathcal{H}_{[p,.]\neq0}$ be the set of hikes none of whose connected components commute with $p$. Clearly, for all $h\in \mathcal{H}_{[p,.]\neq0}$, $hp$ is a walk multiple of $p$ and we need only determine $\mathcal{H}_{[p,.]\neq0}$ to obtain the lemma. But this set is the set of all hikes minus the set of hikes such that at least one connected component commutes with $p$. Such a component must be divisible on the right by prime(s) $p'$ commuting with $p$, hence $\mathcal{V}(p')\subseteq G\backslash p$. Let $\mathcal{P}_{G\backslash p}$ be the set of all such primes. The exclusion-inclusion principle then yields
$
\sum_{h\in\mathcal{H}_{[p,.]\neq0}}h =\sum_{d\in\mathcal{P}_{G\backslash p}^{s.a}}\mu(d)M(d),
$
with the convention that $1\in \mathcal{P}_{G\backslash p}^{s.a}$ and  $M(d)$ is the formal series of the left-multiples of $d$. For any hike this is $M(h)=\det(\mathsf{I}-\mathsf{W})^{-1}h$ since all hikes multiplied by $h$ on the right are left multiples of $h$. Then
$$
\sum_{h\in\mathcal{H}_{[p,.]\neq0}}h = \frac{\sum_{d\in\mathcal{P}_{G\backslash p}^{s.a}}\mu(d) d}{\det(\mathsf{I}-\mathsf{W})} = \frac{\det(\mathsf{I}-\mathsf{W}_{G\backslash p})}{\det(\mathsf{I}-\mathsf{W})}.
$$
The series $W_p$ is obtained upon multiplying the above by $p$ on the right. 
\end{proof}

The sieve based proof of Viennot's lemma suggests a wider family of results to count exactly or asymptotically families of hikes satisfying chosen properties on finite graphs. Concentrating on asymptotic expansions, we have:
\begin{namedtheorem}[Finite Sieve Theorem]
Let $G$ be a finite (weighted di)graph with adjacency matrix $\mathsf{A}$. Let $H$ be an induced subgraph of $G$ and let $\mathcal{P}_H$ be the set of primes on $H$. Let $\rho:\mathcal{H}\mapsto \mathbb{R}$ be a rank function on hikes 
such that $|\mathcal{H}_\varrho|=\lambda^{\varrho}f(\varrho)$ with $\lambda$ a real constant and $f(.)$ a bounded function such that $\lim_{\varrho\to\infty} f(\varrho)$ exists. 

Then the number (weight)\footnote{The notation $S(\mathcal{H}_\rho, \mathcal{P}_H)$ for this quantity is employed in keeping with conventions from sieve theory.} $S(\mathcal{H}_\varrho, \mathcal{P}_H)$ of \underline{hikes} of rank $\rho(h)=\varrho$ which are not multiples of primes on $H$ is asymptotically given by
$$
\frac{S(\mathcal{H}_\varrho, \mathcal{P}_H)}{|\mathcal{H}_\varrho|}\sim \sum_{d\in\mathcal{P}^{s.a}_H}\mu(d)\, \lambda^{-\rho(d)},~~\text{as}~~\varrho\to\infty.
$$
\end{namedtheorem}

The Finite Sieve Theorem's most important application here will be with the length rank-function and sieving subgraph $H=G\backslash p$ for $p$ a prime. This provides the asymptotic expansion of Viennot's lemma:

\begin{namedtheorem}[Length corollary]
Let $G$ be a finite (weighted di)graph with adjacency matrix $\mathsf{A}$ and dominant eigenvalue $\lambda$, which we assume to be unique.\footnote{The theorem extends if $\lambda$ is not unique upon replacing $\lambda^{-1}$ by $\lambda^{-g}$ with $g$ its multiplicity.} Let $p$ be a simple cycle or a simple path on $G$ of length $\ell(p)$ and let $S(\mathcal{H}_l, \mathcal{P}_{G\backslash p})$ be defined as in the Finite Sieve Theorem. 

Then $S(\mathcal{H}_l, \mathcal{P}_{G\backslash p})$ is equal to the number (weight) of \underline{closed walks} of length $\ell(w)=l$ on $G$ whose unique right prime divisor is $p$ and is asymptotically given by
$$
\frac{S(\mathcal{H}_l, \mathcal{P}_{G\backslash p})}{|\mathcal{H}_l|}\sim \frac{1}{\lambda^{\ell(p)}}\det\left(\mathsf{I}-\frac{1}{\lambda}\mathsf{A}_{G\backslash p}\right)~~\text{as}~~\ell\to\infty.
$$
Let $\mathrm{Err}(\mathcal{H}_l,\mathcal{P}_{G\backslash p})$ be the difference between the two terms above. Let $f(l):=[z^l]\zeta(z/\lambda)$ be the coefficient of  $z^l$ in the expansion of $\zeta(z/\lambda)$. Then $f$ is bounded, $\lim_{l\to\infty} f(l)$ exists, and
\begin{align*}
\mathrm{Err}(\mathcal{H}_l,\mathcal{P}_{G\backslash p})=\frac{1}{\lambda^{\ell(p)}}\sum_{k\geq 0}^\infty\left(\frac{\nabla^k[f]\big(l-\ell(p)\big)}{f(l)\,\lambda^k\,k!}-\delta_{k,0}\right)\,\det\!^{(k)}\!\Big(\mathsf{I}-\frac{1}{\lambda}\mathsf{A}_{G\backslash p}\Big).
\end{align*}
with $\delta_{k,0}$ the Kronecker delta and $\nabla$ is the backward difference operator\footnote{This operator acts on a function $F$ of a variable $x$ as $\nabla[F](x):=F(x)-F(x-1)$, and $\nabla^k[F](x)$ designates the $k$th iteration of operator $\nabla$, i.e. $\nabla^k[F](x)=\sum_{j=0}^k (-1)^j\binom{k}{j}F(x-j)$}. Here, $\det\!^{(k)}(\mathsf{I}-\frac{1}{\lambda}\mathsf{A}_{G\backslash p})$ stands for the $k$th derivative of $\det(\mathsf{I}-z\mathsf{A}_{G\backslash p})$ evaluated in $z=1/\lambda$.
\end{namedtheorem}

\begin{proof}[Proof of the Finite Sieve Theorem]
The proof relies on an inclusion-exclusion principle in the poset of the hikes ordered by right-divisibility, see \cite{SIAM2017} for an overview of this poset. 
Let $\mathcal{P}\subsetneq\H$ be a set of primes and $\mathcal{P}^{\text{s.a.}}$ the set of all self-avoiding hikes constructible from $\mathcal{P}$. Let $S(\mathcal{H}_\varrho,\mathcal{P})$ be the number (weight) of hikes in $\mathcal{H}_\varrho$ which are not right-divisible by any prime of $\mathcal{P}$.
The inclusion-exclusion principle--here the extension to hikes of the sieve of Erathostenes-Legendre--yields 
$$
S(\mathcal{H}_{\varrho},\mathcal{P}) = \sum_{d\in \mathcal{P}^{\text{s.a.}}} \mu(d) |\mathcal{M}_d|,
$$
with $|\mathcal{M}_d|$ the number of multiples of $d$ in $\mathcal{H}_\rho$ and
 $\mu(d)$ is the M\"{o}bius function on hikes.\\[-.7em] 

In order to progress, we seek a multiplicative function $\text{prob}(.)$ such that $|\mathcal{M}_d| = \text{prob}(d)|\mathcal{H}_\varrho| + r(d)$. In this expression, $\text{prob}(d)$ approximates the probability that a hike taken uniformly at random in $\mathcal{H}_\varrho$ is right-divisible by $d$. If edge-weights are present, the hikes are not all uniformly probable but follow a distribution dependent on these weights. No knowledge of this distribution is required here. Similarly, $m(d)=\text{prob}(d)|\mathcal{H}_{\varrho}|$ is the expected number of multiples of $d$ in $\mathcal{H}_\varrho$. Finally, $r(d)$ is the associated error term, arising from the fact that $|\mathcal{M}_d|$ is not truly multiplicative. 
Supposing that we can identify the $m(.)$ function, we would obtain 
 $$
S(\mathcal{H}_\varrho,\mathcal{P}) = \sum_{d\in \mathcal{P}^{\text{s.a.}}} \mu(d) m(d)  + \sum_{d\in \mathcal{P}^{\text{s.a.}}} \mu(d) r(d).
$$
Contrary to number theory, the first term does not admit any simpler form without further assumptions on $\mathcal{P}$. This is because of the possible lack of commutativity between some elements of $\mathcal{P}$. We note however that since $\mu(d)$ is non-zero if and only if $d$ is self-avoiding, and since we have required that $m(.)$ be multiplicative,\footnote{But not necessarily totally multiplicative.} then it follows that the first term is determined solely from the values of $m(.)$ over the primes of $\mathcal{P}$.\\

We therefore turn to determining $m(p)$ for $p$ prime. 
The set of left-multiples of $p$ in $\H$ is $\mathcal{M}_p:=\{hp,~h\in \H\}$, hence in bijection with the set $\H$. Thus, the number of left-multiples of $p$ in $\mathcal{H}_{\varrho}$, is exactly $|\mathcal{H}_{\varrho-\rho(p)}|$. 
Then
$$
\text{prob}(p) + \frac{r(p)}{|\mathcal{H}_{\varrho}|} = \frac{|\mathcal{H}_{\varrho-\rho(p)}|}{|\mathcal{H}_{\varrho}|}.
$$
Seeking the best possible probability function $\text{prob}(\rho)$, let us suppose that we can choose this function  such that the error term of the above equation vanishes in the limit $\rho\to\infty$. If this is true, then we obtain 
$$
\text{prob}(\rho) = \lim_{\rho\to \infty}\frac{|\mathcal{H}_{\varrho-\rho(p)}|}{|\mathcal{H}_{\varrho}|}.
$$
In order to progress, we have to make an important assumption regarding the cardinality of the set $\mathcal{H}_{\varrho}$:

\begin{assumption}\label{assumptionAG}
There exists a scaling constant $\lambda$ and bounded function $f:\mathbb{R}\mapsto \mathbb{R}$ such that $\lim_{\varrho\to\infty}f(\varrho)$ exists and for $\varrho\in\mathbb{N}^*$
\begin{equation*}
|\mathcal{H}_{\varrho}|=\lambda^{\varrho} f(\varrho).
\end{equation*}
%
\end{assumption}
\noindent In the case of the length rank function, this assumption is actually a proposition: 
\begin{proposition}\label{LengthScaling}
Let $G$ be a finite (weighted di)graph with dominant eigenvalue $\lambda$ of multiplicity $g$. Let $\mathcal{H}_{\ell}:=\{h\in\mathcal{H}:~\ell(h) =  \ell\}$ be set of all hikes on $G$ of length $\ell$. Then, there exists a bounded function $f:\mathbb{N}\mapsto \mathbb{R}$ such that $\lim_{\ell\to\infty}f(\ell)$ exists and for $\ell\in\mathbb{N}^*$ we have exactly
\begin{equation*}
|\mathcal{H}_{\ell}|=\lambda^{g\ell} f(\ell).
\end{equation*}
\end{proposition}
\begin{proof}
This follows directly from the ordinary zeta function on hikes $\zeta(z)=\det(\mathsf{I}-z\mathsf{A})^{-1}$, from which we have
\begin{equation*}
|\mathcal{H}_{\ell}|=[z^\ell]\!\left(\frac{1}{\det(\mathsf{I}-z\mathsf{A})}\right)
=\sum_{i_1,\cdots,\, i_N\vdash \ell} \lambda^{i_1}_{1}\lambda^{i_2}_{2}\cdots \lambda^{i_N}_{N}=\lambda^\ell\!\! \sum_{i_1,\cdots,\, i_N\vdash \ell} \lambda^{i_1-\ell}\lambda^{i_2}_{2}\cdots \lambda^{i_N}_{N}
\end{equation*}
where the sums run over all non-negative values of $i_j \geq 0$ such that $\sum_j i_j = \ell$ and $\lambda\equiv \lambda_1$ is the eigenvalue of the graph with the largest absolute value. We assume for the moment that $\lambda$ is unique and let 
$
f(\ell):= \sum_{i_1,\cdots, \,i_N\vdash \ell} \lambda^{i_1-\ell}\lambda^{i_2}_{2}\cdots \lambda^{i_N}_{N}=[z^\ell]\zeta(z/\lambda). 
$
This function is clearly bounded and 
$$
\lim_{\ell\to\infty} f(\ell) = \lim_{z\to1/\lambda^-}(1-z\lambda)\zeta(z),
$$ 
exists and is finite. If $|\lambda|$ is not unique and has multiplicity $g$, then the scaling constant for the number of hikes becomes $\lambda^g$ and then $f(\ell)=[z^\ell]\zeta(z/\lambda^g)$. \end{proof}

Proceeding with Assumption~\ref{assumptionAG}--or in the case of the length rank function Proposition~\ref{LengthScaling}--the existence of the limit for $f$ gives 
$$
\text{prob}(p) = \lim_{\varrho\to \infty}\frac{\lambda^{\varrho-\rho(p)} f\big(\varrho-\rho(p)\big)}{\lambda^\varrho f(\varrho)} = \lambda^{-\rho(p)}.
$$
The prob(.) function is multiplicative over the primes as desired and yields $m(p)  = |\mathcal{H}_{\varrho}| \lambda^{-\rho(p)}$. The associated error term is 
\begin{align*}
r(\varrho) = |\mathcal{H}_{\varrho-\rho(p)}| -  |\mathcal{H}_{\varrho}| \lambda^{-\rho(p)} &= \lambda^{\varrho-\rho(p)}\Big(f\big(\varrho-\rho(p)\big)-f(\varrho)\Big).
\end{align*}
%
To establish the validity of these results, we need only verify that they are consistent with our initial supposition concerning the error term, namely that $r(p)/|\mathcal{H}_{\varrho}|$ vanishes in the limit $\varrho\to\infty$. The existence of the limit of $f$ implies $\lim_{\varrho\to\infty}|f\big(\varrho-\rho(p)\big)-f(\varrho)|=0$ and therefore that 
$$
\lim_{\varrho\to\infty }\frac{r(p)}{ |\mathcal{H}_{\varrho}|}=\lim_{\varrho\to\infty }\,\lambda^{-\rho(p)}\Big(f\big(\varrho-\rho(p)\big)-f(\varrho)\Big) =0,
$$
as required.\\

We are now ready to proceed with general self-avoiding hikes. Let $d=p_1\cdots p_{\Omega(d)}$ be self-avoiding. Since $m$ is multiplicative and the rank function is totally additive over $\mathcal{H}$, $m(d) = \prod_i m(p_i) = \lambda^{-\sum_i \rho(p_i)} = \lambda^{-\rho(d)}$. The associated error term follows as
$$
r(d) = |\mathcal{H}_{\varrho-\rho(d)}| -  |\mathcal{H}_{\varrho}| \lambda^{-\rho(d)} = \lambda^{\varrho-\rho(d)}\big(f\big(\varrho-\rho(d)\big)-f(\varrho)\big).
$$ 
Inserting these forms for $m(d)$ and $r(d)$ in the sieve yields 
 \begin{equation}\label{GeneralSieve}
S(\mathcal{H}_{\varrho},\mathcal{P}) = |\mathcal{H}_{\varrho}|\sum_{d\in \mathcal{P}^{\text{s.a.}}} \mu(d) \lambda^{-\rho(d)}  + \lambda^{\varrho}\sum_{d\in \mathcal{P}^{\text{s.a.}}} \mu(d) \lambda^{-\rho(d)} \big(f(\varrho-\rho(d))-f(\varrho)\big).
\end{equation}
We can now progress much further on making an additional assumption concerning the nature of the prime set $\mathcal{P}$. We could consider two possibilities: i) that  $\mathcal{P}$ is the set of all primes on an induced subgraph $H\prec G$; or ii) that $\mathcal{P}$ is a cut-off set, e.g. one disposes of all the primes of length $\ell(p)\leq \Theta$. 

In the situation where all primes commute with one-another--i.e. when Viennot's theory of heaps of pieces reduces to number theory \cite{SIAM2017}--then one may have both i) and ii) simultaneously. This is because in the semi-commutative extension of number theory that holds on the monoid of hikes, coprimality extends to being vertex-disjoint \cite{SIAM2017}. Therefore, requiring all primes to commute is equivalent to forcing the graph $G$ to be made of disjoint oriented simple cycles. Consequently, we can choose $H$ to be the induced subgraph of $G$ comprising all simple cycles of length up to some cut-off $\Theta$ and both situations i) and ii) are realised. For this reason, number theoretic sieves benefit from the advantages of both situations: ii) guarantees that sieves can be used to obtain estimates on the number of primes, while i) allows these estimates to be computable. On general graphs however, i) and ii) are not compatible and while ii) could be used to obtain direct estimates for the number of primes of any length, a problem of great interest, this actually makes the sieve NP-hard to implement. We therefore focus on the first situation.\\

Let $H\prec G$ be an induced subgraph of the graph $G$ and let that $\mathcal{P}\equiv \mathcal{P}_H$ be the set of all primes (here simple cycles) on $H$. To conclude the proof we need only show that the error term of Eq.~(\ref{GeneralSieve}) is asymptotically dominated by the first term $\sum_{d\in \mathcal{P}_H^{\text{s.a.}}} \mu(d) \lambda^{-\rho(d)}$. To this end, we note that since $H$ is finite\footnote{$G$ is finite and so are all its induced subgraphs.} 
$$
\lambda^{\varrho}\sum_{d\in \mathcal{P}_H^{\text{s.a.}}} \mu(d) \lambda^{-\rho(d)} \big(f(\varrho-\rho(d))-f(\varrho)\big),
$$
is a sum involving finitely many self-avoiding hikes $d$. In addition,
given that $\lim_{\varrho\to\infty}f(\varrho)$ exists (either by Assumption~\ref{assumptionAG} or by Proposition~\ref{LengthScaling} for the length rank function), $\lim_{\varrho\to\infty}f(\varrho-\rho(d))-f(\varrho) =0$ as long as $\rho(d)$ is finite, which is guaranteed by the finiteness of $H$. We have consequently established that the error term comprises finitely many terms, each of which vanishes in the $\varrho\to\infty$ limit. As a corollary, the first term is asymptotically dominant: 
 \begin{equation*}
\frac{S(\mathcal{H}_{\varrho},\mathcal{P}_H)}{|\mathcal{H}_{\varrho}|} \sim \sum_{d\in \mathcal{P}^{\text{s.a.}}_H} \mu(d) \lambda^{-\rho(d)}  ~~\text{as}~~\varrho\to\infty,
\end{equation*}
where we assume that $|\mathcal{H}_{\varrho}|\neq0$.

We now turn to establishing the length corollary of the Finite Sieve Theorem. We are specifically looking for the number of closed walks which are multiples of a prime $p$.  To this end, we  need only choose $H$ correctly. Let $h$ be a hike, for $w=hp$ to be a walk of length $l$, then $h$ must have length $l-\ell(p)$ and be such that none of its right-prime divisor commutes with $p$. The sieve must thus eliminate all hikes $h$ which are left-multiples of primes \emph{commuting} with $p$. Observe that all such primes are on $H=G\backslash p$. Consequently the Finite Sieve Theorem yields, for $|\mathcal{H}_{l-\ell(p)}|\neq 0$,
 \begin{align*}
S(\mathcal{H}_l, \mathcal{P}_{G\backslash p})&=|\mathcal{H}_{l-\ell(p)}| \sum_{d\in \mathcal{P}^{\text{s.a.}}} \mu(d) \lambda^{-\ell(d)}\\  
&\hspace{7mm}+ \lambda^{l-\ell(p)}\sum_{d\in \mathcal{P}^{\text{s.a.}}} \mu(d) \lambda^{-\ell(d)} \big(f(l-\ell(p)-\ell(d))-f(l-\ell(p))\big),\nonumber
\end{align*}
where $\lambda$ is now the graph dominant eigenvalue per Proposition~\ref{LengthScaling}. The asymptotically dominant term is a sum over all the self-avoiding hikes on $G\backslash p$, each with coefficient $\mu(d)\lambda^{-\ell(d)}$ and is equal to $\det(\mathsf{I}-\lambda^{-1}\mathsf{A}_{G\backslash p})$. Since furthermore $|\mathcal{H}_{l-\ell(p)}|=|\mathcal{H}_{l}|\lambda^{-\ell(p)}f\big(l-\ell(p)\big)/f(l)$, we have asymptotically for $l\gg1$
$$
\frac{S(\mathcal{H}_l, \mathcal{P}_{G\backslash p})}{|\mathcal{H}_{l}|}\sim\lambda^{-\ell(p)}\det\left(\mathsf{I}-\frac{1}{\lambda}\mathsf{A}_{G\backslash p}\right),
$$
while the error terms is 
\begin{align*}
\text{Err}(\mathcal{H}_l,\mathcal{P}_{G\backslash p})&:=\frac{S(\mathcal{H}_l, \mathcal{P}_{G\backslash p})}{|\mathcal{H}_{l}|}-\lambda^{-\ell(p)}\det\left(\mathsf{I}-\frac{1}{\lambda}\mathsf{A}_{G\backslash p}\right),\\
&=\lambda^{-\ell(p)}\big(f\big(l-\ell(p)\big)/f(l)-1\big)\det\left(\mathsf{I}-\frac{1}{\lambda}\mathsf{A}_{G\backslash p}\right)+\\
&\hspace{10mm}\frac{\lambda^{-\ell(p)}}{f(l)}\sum_{d\in \mathcal{P}^{\text{s.a.}}} \mu(d) \lambda^{-\ell(d)} \big(f(l-\ell(p)-\ell(d))-f(l-\ell(p))\big).
\end{align*}
The last line can be brought in determinantal form as well, since
$$
f(l-\ell(p)-\ell(d))-f(l-\ell(p)) = \sum_{k\geq 1}^{\ell(d)} \frac{\nabla^k[f]\big(l-\ell(p)\big)}{k!}\big(\ell(d)\big)_{(k)},
$$
with $(a)_{(k)}:=\prod_{i=0}^{k-1}(a-i)$ the falling factorial and $\nabla$ the backward difference operator. This operator acts on a function $F$ of a variable $x$ as $\nabla[F](x):=F(x)-F(x-1)$, and $\nabla^k[F](x)$ designates the $k$th iteration of operator $\nabla$, i.e. $\nabla^k[F](x)=\sum_{j=0}^k (-1)^j\binom{k}{j}F(x-j)$. Then
\begin{align*}
&\frac{\lambda^{-\ell(p)}}{f(l)}\sum_{d\in \mathcal{P}^{\text{s.a.}}} \mu(d) \lambda^{-\ell(d)} \big(f(l-\ell(p)-\ell(d))-f(l-\ell(p))\big).\\
&\hspace{10mm}=\lambda^{-\ell(p)}\sum_{d\in \mathcal{P}_H^{\text{s.a.}}} \mu(d) \lambda^{-\ell(d)} \sum_{k\geq 1}^{\ell(d)} \frac{\nabla^k[f]\big(l-\ell(p)\big)}{f(l)\,k!}\big(\ell(d)\big)_{(k)}.
\end{align*}
The upper limit of the inner sum over $k$ can be extended to $\infty$ since all terms with $k>\ell(d)$ are nul. Noting that $\big(\ell(d)\big)_{(k)}z^{\ell(d)}=z^k\frac{d^k}{dz^k}z^{\ell(d)}$, this allows us to invert the two sums, yielding
\begin{align*}
\lambda^{-\ell(p)}\sum_{k\geq 1}^\infty\frac{\nabla^k[f]\big(l-\ell(p)\big)}{f(l)\lambda^k\,k!}\,\det\!^{(k)}\!\Big(\mathsf{I}-\frac{1}{\lambda}\mathsf{A}_{G\backslash p}\Big).
\end{align*}
Now setting $k=0$ in the above recovers $\lambda^{-\ell(p)}\big(f\big(l-\ell(p)\big)/f(l)-1\big)\det\left(\mathsf{I}-\frac{1}{\lambda}\mathsf{A}_{G\backslash p}\right)$ with the exception of the $-1$ in the parenthesis, which can be introduced as $-\delta_{k,0}$.
This establishes the Finite Sieve Theorem and its length corollary.
\end{proof}


\section{Infinite graphs}\label{InfiniteG}
\subsection{Viennot's lemma on infinite graphs}\label{ViennotInfinite}
As we have seen, when counting closed walk multiples of a prime according to their length on finite graphs\footnote{Recall that this is the same thing as counting closed walk whose last erased loop is a certain simple cycle.}, the Finite Sieve Theorem produces the asymptotics of Viennot's lemma. To put the infinite graphs results in context we thus start by proving that this lemma extends to infinite graphs with bounded degree. Here, we state only the ordinary generating function version of the extension here, that for formal series on hikes is provided in Appendix~\ref{ViennotFormal}.  

\begin{figure}
\begin{center}
\includegraphics[width=1\textwidth]{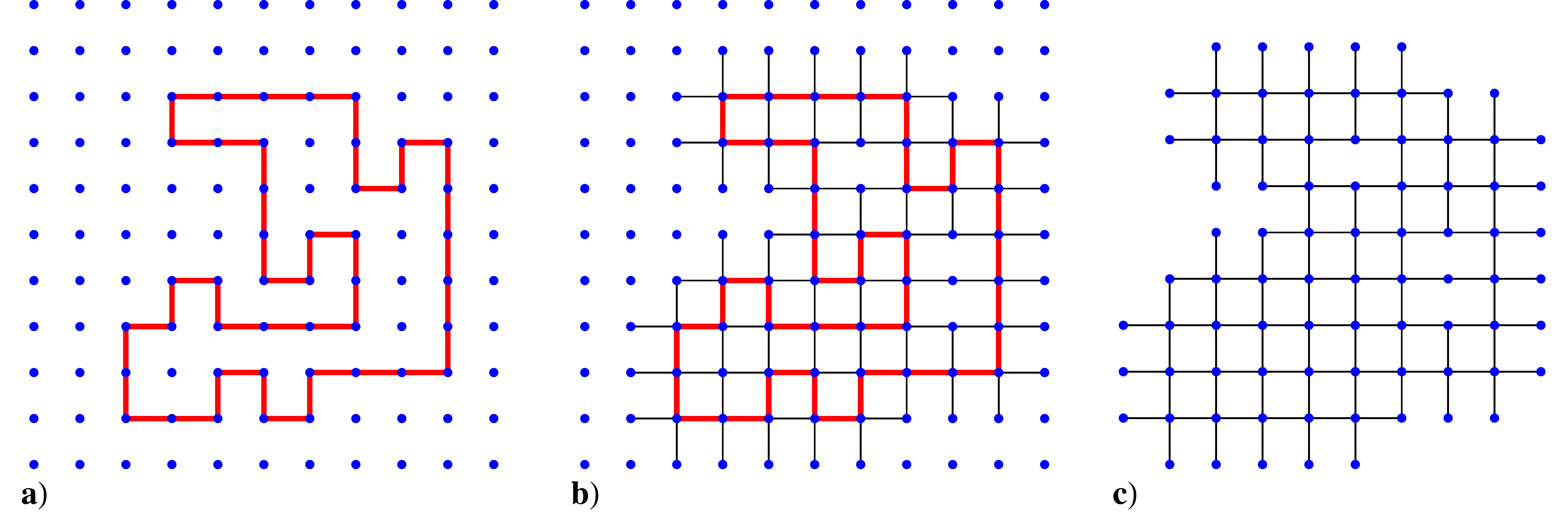}
  \end{center}
  \vspace{-4.5mm}
 \caption{\label{GraphGP}{\footnotesize \textbf{a)} A self-avoiding polygon $p$ on the square lattice; \textbf{b)} Polygon $p$ and all the edges that have at least one endpoint on $p$; \textbf{c)} Graph $G_p$.}}
  \vspace{-2mm}
\end{figure}

\begin{proposition}\label{InfiniteViennot}
Let $G$ be an infinite (weighted di)graph with bounded degree. Let $p$ be a prime on $G$, with support $\mathcal{V}(p)$ and neighborhood $\mathcal{N}(p)$. Let $G_p\prec G$ be the induced subgraph of $G$ with vertex set $\mathcal{V}(G_p)=\mathcal{V}(p)\cup\mathcal{N}(p)$ (see Fig.~\ref{GraphGP} for an example) and $\mathsf{B}_p$ its adjacency matrix. Let $\mathsf{R}(z)=\big(\mathsf{I}-z\mathsf{A}\big)^{-1}\big|_{G_p}$ be the restriction of the resolvent of $G$ to $G_p$.  
%
%

Then the ordinary generating function $R_p(z)$ of closed walks whose last erased loop is $p$ is given by
\begin{equation}\label{ViennotLemmaInfinite}
R_p(z)=z^{\ell(p)}\det\!\big(\mathsf{I}+z\mathsf{R}(z)\mathsf{B}_p\big).
\end{equation}
\end{proposition}

Observe that neither the series $\det(\mathsf{I}-z\mathsf{A}_{G\backslash p})$ nor $\det(\mathsf{I}-z\mathsf{A})$ appearing in Viennot's original lemma are well defined on infinite graphs---e.g. all their finite order coefficients can be infinite. Their ratio evaluated on a sequence of finite graphs converging to $G$ (as defined in Appendix~\ref{GraphSeq}) nonetheless gives rise to a well defined series in the sense of the Proposition above. 

\begin{proof}
Let $\{G^{\text{Tor}}_N\}_{N\in\mathbb{N}}$ be the small torus sequence of graphs converging to $G$ as defined in Appendix~\ref{GraphSeq}. For any SAP $p$, define $N(p)\in\mathbb{N}$ such for all $N\geq N(p)$, then $G_p$ is an induced subgraph of $G_N$.\footnote{Existence of $N(p)$ is guaranteed for finite length SAPs as the small torus $G^{\text{Tor}}_{N:=n^2}$ contains the ball of radius $n/2$ centred on the starting point of $p$ on $G$.} 
The result follows by using Viennot's lemma on $G^{\text{Tor}}_N$ and grouping all terms into a single determinant. 
Let $p$ be a prime of finite support, hence finite length $\ell(p)$ on $G$. 
Let $N\geq N(p)$, $\mathsf{R}_N(z):=(\mathsf{I}-z\mathsf{A}_{G^{\text{Tor}}_N})^{-1}$ and $\mathsf{B}_N(p):=\mathsf{A}_{G^{\text{Tor}}_N}-\mathsf{A}_{G^{\text{Tor}}_N\backslash p}$. Then Viennot's lemma on $G^{\text{Tor}}_N$ yields the ordinary generating function $R_{N,p}(z)$ of closed walks whose last erased loop is $p$ on $G^{\text{Tor}}_N$ as
\begin{align*}
R_{N,p}(z)=\frac{\det(\mathsf{I}-z\mathsf{A}_{G^{\text{Tor}}_N\backslash p})}{\det(\mathsf{I}-z\mathsf{A}_{G^{\text{Tor}}_N})} &=\det\!\big(\mathsf{R}_N(z)\big)\det\big(\mathsf{I}-z\mathsf{A}_{G^{\text{Tor}}_N}+z\mathsf{B}_{N,p}\big),\\
&=\det\!\big(\mathsf{I}+z\mathsf{R}_N(z)\mathsf{B}_{N,p}\big).
\end{align*}
Given that the sequence of small tori $G^{\text{Tor}}_N$ converges to $G$, $\lim_{N\to\infty} R_{N,p}(z)=R_p(z)$ provided we can show that $\det\!\big(\mathsf{I}+z\mathsf{R}_N(z)\mathsf{B}_{N,p}\big)$ is well defined under this limit. 

Now since $\big(\mathsf{B}_{N,p}\big)_{ij}=0$ unless both $i,j\in\mathcal{V}(p)\cup \mathcal{N}(p)$ and since $p$ is of finite length, then $\det\big(\mathsf{I}+z\mathsf{R}_N(z)\mathsf{B}_{N,p}\big)$ is equal to the determinant of the finite $|\mathcal{V}(p)\cup \mathcal{N}(p)|\times |\mathcal{V}(p)\cup \mathcal{N}(p)|$ matrix $\mathsf{Q}_{ij}:=\big(\mathsf{I}+z\mathsf{R}_N(z)\mathsf{B}_{N,p}\big)_{ij}$, for $i,j$ in $p$ and its neighborhood. Given that $\mathsf{B}_{N,p}\big|_{G_p}=\mathsf{B}_p$ for all $N\geq N(p)$ and since the graph has bounded degree $\lim_{N\to\infty}\mathsf{R}_N(z) = \mathsf{R}(z)$ is well defined. We consequently have
$$
R_p(z) = \lim_{N\to\infty}R_{N,p}(z)=\lim_{N\to\infty} \det\!\Big(\big(\mathsf{I}+z\mathsf{R}_N(z)\mathsf{B}_{N,p}\big)|_{G_p}\Big)=\det\!\big(\mathsf{I}+z\mathsf{R}(z)\mathsf{B}_p\big).
$$
We emphasise that the determinant $\det\!\big(\mathsf{I}+z\mathsf{R}(z)\mathsf{B}_p\big)$ is equal to that of the finite matrix $\mathsf{Q}$ and no considerations pertaining to the determinants of infinite matrices is needed.  
\end{proof}

We illustrate Viennot's lemma on infinite graphs with the ordinary generating function $R_e(z)$ of closed walks whose last erased loop is an edge cycle $e$ on the square lattice. Direct application of Eq.~(\ref{ViennotLemmaInfinite}) gives
\begin{align}\label{Redge}
R_e(z) &= \frac{\pi}{4} - \frac{1}{16}+\frac{1}{16 \pi^2}\big(64 z^2-4\big) K\!\big(16 z^2\big)^2+\frac{1}{4 \pi} \Big(K\!\big(16 z^2\big)-\pi ^2\Big),\\
   &=z^2+7 \,z^4+70 \,z^6+807 \,z^8+10,046 \,z^{10}+131,206 \,z^{12}+\cdots\nonumber
\end{align}
\begin{figure}
\begin{center}
\includegraphics[width=1\textwidth]{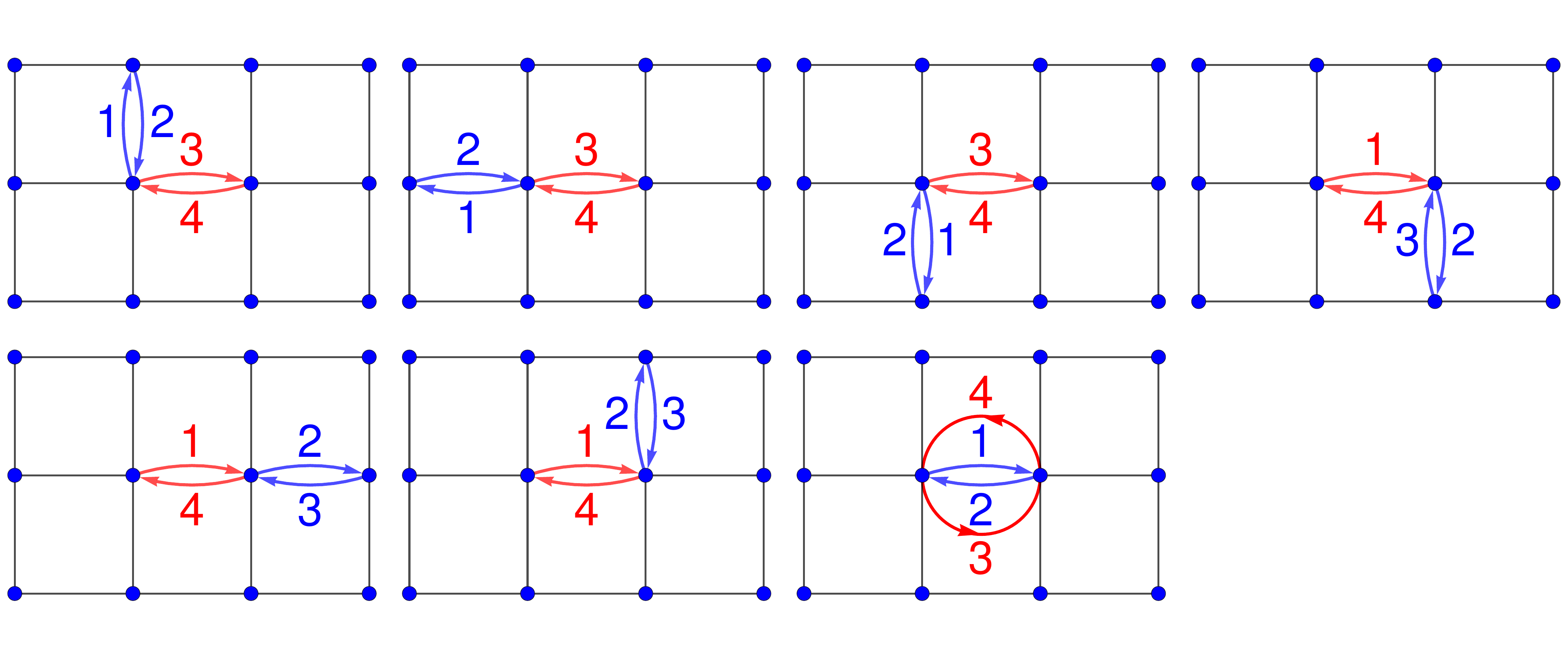}
  \end{center}
  \vspace{-9.5mm}
 \caption{\label{EdgeMultiples}{\footnotesize Illustration of the 7 walks of length 4 on the square lattice whose last erased loop is an edge cycle (highlighted in red). The numbers next to the highlighted edges indicate the order in which these edges are traversed.}}
  \vspace{-2mm}
\end{figure}
where $K(x):=\int_0^{\pi/2}\big(1-x\sin^2(\theta)\big)^{\!-1/2}d\theta$ is the complete elliptic integral of the first kind. In Fig.~(\ref{EdgeMultiples}) we illustrate the 7 walks of length 4 on the square lattice whose last erased loop is an edge cycle, as correctly counted by  $[z^4]R_e(z)=7$. 
Typically, generating functions for walks whose last erased loop is a certain self-avoiding polygon $p$, such as $R_e(z)$ above, are not meromorphic as functions of $z$. Consequently, the asymptotic growth of their coefficients cannot be determined with the traditional tools of meromorphic asymptotics \cite{flajolet2009analytic}.  
In the next section we develop a generic method as a replacement that allows us to determine this asymptotic growth for any prime $p$. 

\subsection{Asymptotic expansion of Viennot's lemma}
We now aim at establishing the formula for the fraction of walks whose last erased loop is any chosen self-avoiding polygon on any infinite vertex-transitive graph.


\begin{proof}[Proof of Corollary~\ref{InfiniteSieve}]
The first difficulty in extending the Finite Sieve Theorem to infinite graphs comes from the proliferation of hikes on such graphs: there are either exactly 0 or infinitely many hikes of any given length. Furthermore, the number of hikes increases uncontrollably with the length as there are also infinitely many more hikes of any length $L'>L$ than of length $L$.
These observations continue to be true even when hikes are considered up to translation. 
To make matters worse, the fraction of all hikes which are walks is exactly 0; and even with edge weights uniformly set to $1/\lambda$, the total weight carried by all walks whose last erased loop is any SAP $p$ is still divergent.
 
To resolve these serious difficulties requires us to separate the finite sieve results into two contributions, the first of which relates hikes to walks and the second relates closed walks to walks with fixed last erased loop $p$. This second contribution must itself be dealt with carefully to cure  divergences stemming from the non-meromorphic nature of the generating functions produced by the extension of Viennot's Lemma to infinite graphs. We illustrate every step of the proof with explicit results on the square lattice.\\ 

Let us consider the dominant term of the asymptotic expansion of Viennot's lemma on the sequence of finite graphs $G^{\text{Tor}}_N$. Fixing $N$, consider $p$ a self-avoiding polygon on $G^{\text{Tor}}_N$.
Using the transformation presented in the extension of Viennot's lemma to infinite graphs, we can express the asymptotic fraction of hikes of length $l$ that are walks whose last erased loop is $p$ as $l\to\infty$ as
\begin{align*}
\frac{S(\mathcal{H}_l, \mathcal{P}_{G^{\text{Tor}}_N\backslash p})}{|\mathcal{H}_{l}|}&\sim\lambda^{-\ell(p)}\det\left(\mathsf{I}-\frac{1}{\lambda}\mathsf{A}_{G^{\text{Tor}}_N\backslash p}\right)\\
&\hspace{5mm}=\lambda^{-\ell(p)}\lim_{z\to1/\lambda^{-}}\zeta_N(z)^{-1}\det\left(\mathsf{I}+z\mathsf{R}_N(z)\mathsf{B}_p\right).
\end{align*}
This suggests a strategy consisting of proving separate convergence in $z\to1/\lambda^{-}$ of the two terms in the limit above. This naive strategy ultimately fails, but the procedure that works is best understood once the nature of this failure is made apparent and several results we will obtain along the way are necessary to implement the correct proof strategy. 
In this spirit, we pretend to follow the naive approach and thus first examine the behaviour of the limit $\lim_{z\to\lambda^{-1}}\zeta_N(z)^{-1}$ asymptotically in $N$: 
\begin{lemma}\label{RootedHikesLemma}
Let $\{G^{\text{Tor}}_N\}_{N\in\mathbb{N}}$ be the small tori sequence of vertex-transitive graphs converging to the infinite bounded-degree vertex-transitive graph $G$ with maximum eigenvalue $\lambda$. Let $\zeta_N(z)$ be the zeta function of hikes on $G^{\text{Tor}}_N$ and let $R(z)=\mathsf{R}(z)_{ii}$ be the ordinary generating function of closed walks on $G$. Then,
\begin{equation}\label{ZetaN}
\lim_{N\to\infty} \zeta_N(z)^{1/N} = \tilde{\zeta}(z)=\exp\left(\int \frac{1}{z}\big(R(z)-1\big) dz\right).
\end{equation}
Furthermore $\alpha:=\lim_{z\to1/\lambda^-}\tilde{\zeta}^{-1}(z)$ is well defined. All the coefficients $[z^n]\tilde{\zeta}(z)$ are positive integers, while the coefficients $[z^n]\tilde{\zeta}^{-1}(z)$ are integers.
\end{lemma}

\begin{proof}
Since the log-derivative of the hike zeta function is the trace of the resolvent, the log derivative of $\zeta(z)^{1/N}$ is, on vertex-transitive graphs, a single diagonal entry of the resolvent, yielding Eq.~(\ref{ZetaN}). Assuming that the dominant eigenvalue is unique,  $R(z)$ diverges at worse as $1/(1-z\lambda)$ around $z\sim 1/\lambda$ so that necessarily
$\int \frac{1}{z}\big(R(z)-1\big) dz$ converges in $1/\lambda$.

To understand the coefficients of $[z^n]\tilde{\zeta}(z)$, let $h$ be a hike and let $w_1,\cdots w_n$ be vertex-disjoint walks making up hike $h$, i.e. $h=w_1w_2\cdots w_n$ modulo the fact that all these walks commute with one another. A \emph{rooted hike} $h_{\text{root}}$ is the object obtained from $h$ on translating all walks $w_i$ so that the origin lies on their unique right prime divisor but retaining the fact that they commute. Then $[z^n]\tilde{\zeta}(z)$ counts the number of rooted hikes of length $n$, $\tilde{\zeta}(z)$ is the zeta function of rooted hikes, while $\tilde{\mu}(z):=\tilde{\zeta}^{-1}(z)$ is the M\"obius function on rooted hikes.  
\end{proof}

%

\begin{corollary}
On the square lattice, the generating function of rooted hikes is
\begin{align*}
\tilde{\zeta}(z) &= \exp \left(\sum_{n\geq 1}\binom{2n}{n}^2 \frac{z^{2n}}{2n}\right),\\
&=1+2 z^2+11 z^4+86 z^6+805 z^8+8402 z^{10}+94306 z^{12}+\hdots,
\end{align*}
This is illustrated on Fig.~(\ref{RootedHikes}). The corresponding M\"obius function is
\begin{align*}
\tilde{\mu}(z) &:=1/\tilde{\zeta}(z),\\
&=1-2 z^2-7 z^4-50 z^6-456 z^8-4728 z^{10}-53095 z^{12}+\hdots,
\end{align*}
and $\alpha:=\tilde{\zeta}^{-1}(1/\lambda) = \frac{1}{4} e^{\frac{4 C}{\pi }}\simeq 0.8025...$ with $C$ Catalan's constant. 
\end{corollary}

\begin{figure}[!t]
\begin{center}
\vspace{-3mm}
\includegraphics[width=.8\textwidth]{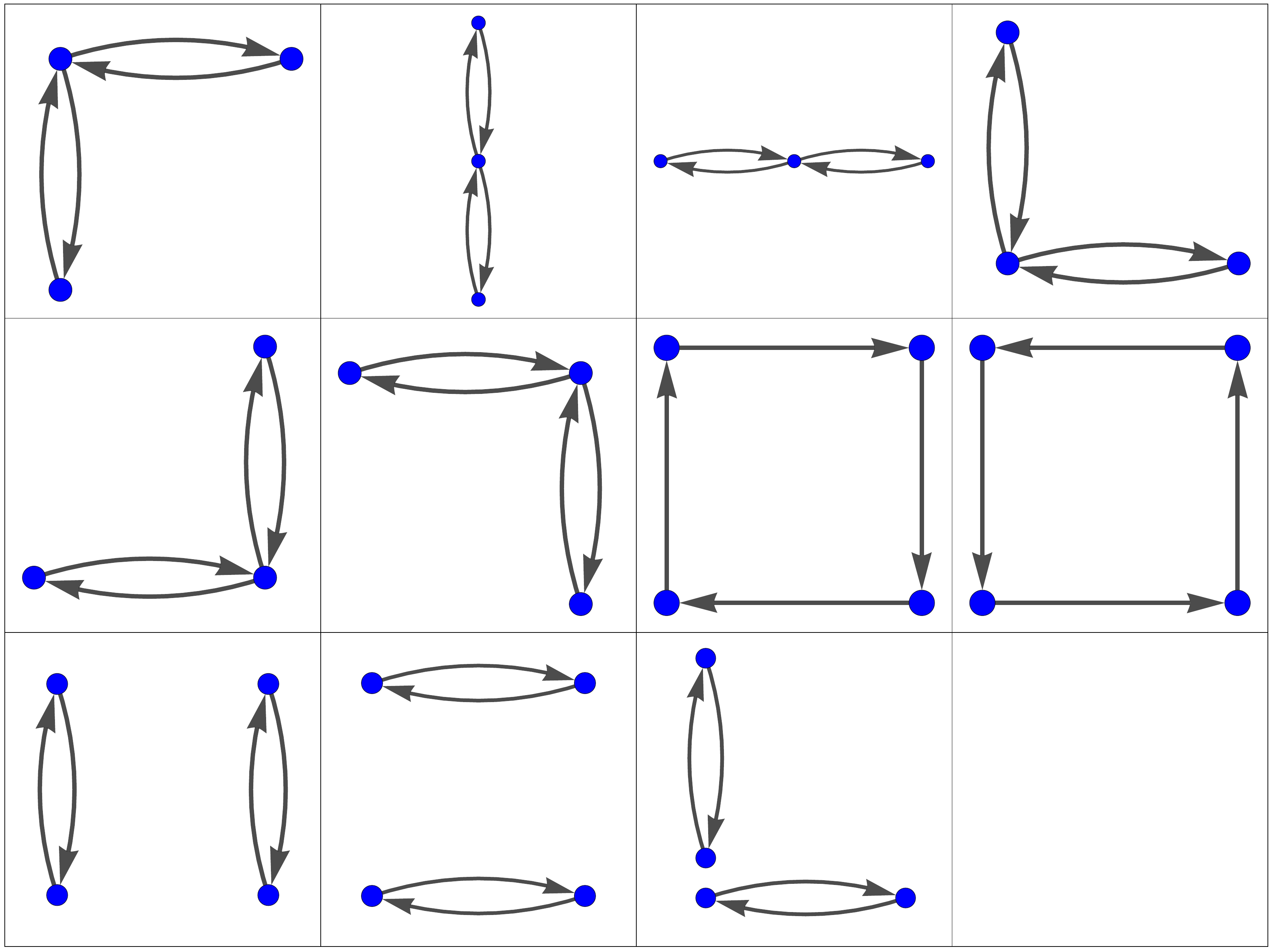}
\caption{The $[z^4]\tilde{\zeta}(z)=11$ objects of length 4 giving rise to distinct rooted hikes of length 4 on the infinite square lattice. Rigorously speaking, rooted hikes are obtained on considering the above objets up to translation of each walk.}
\label{RootedHikes}
\end{center}
\vspace{-5mm}
\end{figure}

\begin{proof}
While this result is an immediate corollary of the precedent lemma, we can prove it directly by considering the sequence of small $n\times n$ square lattices $G^{\text{Sq}}_{n^2}$ converging to $G$ as described in Appendix~\ref{GraphSeq}. The eigenvalues of $G^{\text{Sq}}_{N:=n^2}$ are $\lambda(i,k)=2\cos\Big(\pi \frac{j}{n+1}\Big)+2\cos\Big(\pi \frac{k}{n+1}\Big)$, where $j$ and $k$ are two integers between 1 and $n$. Product-integration yields
$$
\lim_{N\to\infty}\zeta_N(z)^{1/N}= \exp\left(\int_{S} \log\big(1-2z\cos(\pi x)-2z \cos(\pi y)\big) dx dy\right).
$$
with $S$ the unit square $[0,1]\times[0,1]$. We now obtain the relation to $R(z)$ directly. Observe that 
\begin{align*}
&\int_S \log\big(1-2z\cos(\pi x)-2z \cos(\pi y)\big) dx dy = \\
&\hspace{15mm}-\sum_{n\geq 1} \frac{(2z)^{2n}}{2n}\int_S\big( \cos(\pi x)+ \cos(\pi y)\big)^{2n} dx dy,
\end{align*} 
where we used the fact that odd powers of the sum of the cosines must have 0 integral over $S$ since both cosines are symmetric functions on $S$.
Then we have
\begin{align*}
\int_S\big( \cos(\pi x)+ \cos(\pi y)\big)^{2n} dx dy &= \sum_{h=0}^n\binom{2n}{2h}\int_0^1 \cos(\pi x)^{2n-2h}dx \int_0^1 \cos(\pi y)^{2h} dy,\\
&=\sum_{h=0}^n\binom{2n}{2h}\frac{\Gamma \left(n-h+\frac{1}{2}\right)}{\sqrt{\pi } (n-h)!} \frac{\Gamma \left(h+\frac{1}{2}\right)}{\sqrt{\pi } h!},\\
&=\frac{4^n \Gamma \left(n+\frac{1}{2}\right)^2}{\pi  (n!)^2}=\binom{2n}{n}^{\!\!2}4^{-n}.
\end{align*}
It follows that  
$
\int_S \log\big(1-2z\cos(\pi x)-2z \cos(\pi y)\big) dx dy =-\sum_{n=1}^{\infty} \binom{2n}{n}^{\!\!2} \,\frac{z^{2n}}{2n}=-\int \frac{1}{z}\big(R(z)-1\big) dz.
$
\end{proof}

Lemma~\ref{RootedHikesLemma} implies that asymptotically, when $N\gg 1$, $\lim_{z\to1/\lambda^-} \zeta_N(z)^{-1}\sim \alpha^N$ is well defined. We will see later on that  $\alpha^N$ relates the density of hikes to that of walks on an infinite graphs. More precisely, it can be interpreted as follows: every time a vertex is added to a graph $G^{\text{Tor}}_N$ or $G^{\text{Sq}}_N$, $N\gg 1$, the fraction of hikes which are closed walks is multiplied by $\alpha$.\\ 

Returning to the fraction of hikes on $G^{\text{Tor}}_N$ which are walk whose last erased loop is $p$, Lemma~\ref{RootedHikesLemma} and our naive strategy suggest that we express this fraction asymptotically for $N\gg1$ as
\begin{equation}\label{Decompo1}
\det\!\left(\mathsf{I}-\frac{1}{\lambda}\mathsf{A}_{G^{\text{Tor}}_N\backslash p}\right)\stackrel{\,?}{=}\alpha^N\lim_{z\to\lambda^{-1}}\det\!\big(\mathsf{I}+z\mathsf{R}_N(z)\mathsf{B}_p\big).
\end{equation}
In fact the right hand side is divergent under the limit $z\to1/\lambda^-$, but we will see below that its divergence is generic: it is the same for any SAP $p$  
as well as for all closed walks. 
In other terms, the right hand side diverges only because the series of all walks diverges at the point $z\to1/\lambda^-$, hinting that the fraction of walks whose last erased loop is $p$ with respect to all walks might itself be well defined. To put this observation on firm foundations, we need to show that $\lim_{z\to1/\lambda^-}\det\!\big(\mathsf{I}+z\mathsf{R}_N(z)\mathsf{B}_p\big)/R(z)$ is well defined, and that the fraction of all walks with respect to all hikes is well defined as well, in spite of the generic divergence. These two results rely on different techniques. We start with proving of the second fact without which the first would be useless. 

\begin{lemma}\label{WalkstoHikes}
Let $G$ be an infinite vertex-transitive graph of bounded degree. Let $\{G^{\text{Tor}}_N\}_{N\in\mathbb{N}}$ be the small tori sequence of vertex-transitive graphs converging to $G$. Let $\bullet$ be any vertex of $G$. Then the fraction $F_\bullet$ of hikes which are closed walks from $\bullet$ to itself is asymptotically given by 
$$
F_\bullet\sim\frac{\alpha^N}{N},\quad N\to\infty.
$$
In this expression $\alpha$ is the constant defined in Lemma~\ref{RootedHikesLemma}. 
\end{lemma}

The $1/N$ factor in $F_\bullet$ originates from that we have fixed the vertex $\bullet$. If instead we consider translation invariant quantities, i.e. we consider all closed walks irrespectively of their starting point, then the fraction of hikes which are closed walks is asymptotically $\alpha^N$, $N\gg1$.

\begin{remark}
Let $|W|_N(\ell)$ and $|H|_N(\ell)$ be the \emph{total} number of walks and of hikes up to length $\ell\geq 0$ on $G_N^{\text{Tor}}$, respectively. Then, Lemma~(\ref{WalkstoHikes}) is equivalent to 
$$
\lim_{N\to \infty}\lim_{\ell\to\infty}\frac{1}{N}\log\big(|W|_N(\ell)/|H|_N(\ell)\big)=\log(\alpha),
$$
that is the density of the logarithm of the fraction of hikes which are closed walks is well defined on $G$.
\end{remark}

\begin{proof}
On the finite graphs $G^{\text{Tor}}_N$, the fraction of all closed walks off $\bullet$ with respect to hikes is given by $F_\bullet=\det(\mathsf{I}-\frac{1}{\lambda} \mathsf{A}_{G^{\text{Tor}}_N\backslash \bullet})$. Let $\mathsf{B}_{N,\bullet}$ be the $N\times N$ identically zero matrix except on edges adjacent to $\bullet$ where its value is 1, i.e. so that 
$\mathsf{A}_{G^{\text{Tor}}_N\backslash \bullet}= \mathsf{A}_{G^{\text{Tor}}_N}-\mathsf{B}_{N,\bullet}$. Then $F_\bullet=\det(\mathsf{I}-\frac{1}{\lambda} \mathsf{A}_{G^{\text{Tor}}_N}+\frac{1}{\lambda}\mathsf{B}_{N,\bullet})$ and we determine this determinant by expansion around  $\mathsf{A}_{G^{\text{Tor}}_N}$, hereafter denoted $\mathsf{A}_{N}$ in order to alleviate the equations. This gives
\begin{align*}
F_\bullet&=\det\left(\mathsf{I}-\frac{1}{\lambda}\mathsf{A}_N\right)+\frac{1}{\lambda}\operatorname{Tr}\left(\operatorname{adj}\left(\mathsf{I}-\frac{1}{\lambda}\mathsf{A}_N\right)\mathsf{B}_{N,\bullet}\right)\\
&\hspace{10mm}-\frac{1}{2\lambda^2}\,\lim_{z\to \lambda^{-1}}\det\left(\mathsf{I}-z\mathsf{A}_N\right)
\left[\mathrm{Tr}^2\big(\mathsf{R}_N(z) \mathsf{B}_{N,\bullet}\big)-\mathrm{Tr}\big(\mathsf{R}_N(z) \mathsf{B}_{N,\bullet}\big)^2\right].
\end{align*}
All orders of degree higher than two are exactly zero.
In this expression $\mathsf{R}_N(z):=\big(\mathsf{I}-z\mathsf{A}_N\big)^{-1}$, $\mathrm{Tr}(\mathsf{M})$ designates the trace of a matrix $\mathsf{M}$ and $\operatorname{adj}(\mathsf{M})$ its adjugate. 

Order zero of the expansion is null since $\lambda$ is an eigenvalue of $G^{\text{Tor}}_N$. By the same token, the adjugate matrix of the first order is proportional to the projector $\mathsf{P}_{N,\lambda}$ onto the eigenvector associated with $\lambda$ on $G^{\text{Tor}}_N$
$$
\operatorname{adj}\left(\mathsf{I}-\frac{1}{\lambda_N}\mathsf{A}_N\right) = \alpha_N \,\mathsf{P}_{N,\lambda},
$$
where $\alpha_N:=\prod_{\lambda_i\in\mathrm{Sp}(\mathsf{A}_N)\atop \lambda_i\neq \lambda_N}(1-\lambda_i/\lambda_N)$. Since $G^{\text{Tor}}_N$ is regular $\mathsf{P}_{\lambda_N}=\frac{1}{N}\mathsf{J}$, where $\mathsf{J}_{ij}=1$. Then 
$$
\frac{1}{\lambda}\operatorname{Tr}\left(\operatorname{adj}\left(\mathsf{I}-\frac{1}{\lambda}\mathsf{A}_N\right)\mathsf{B}_{N,\bullet}\right)= \frac{\alpha_N}{\lambda}\times \frac{2\lambda}{N}=2\frac{\alpha_N}{N},
$$
since $\mathsf{B}_{N,\bullet}$ has exactly $2\lambda$ non-zero entries. 

The second order of the expansion is always well defined, as is readily seen from the equivalent form
\begin{align}
&-\frac{1}{2\lambda^2}\,\lim_{z\to \lambda^{-1}}\det\left(\mathsf{I}-z\mathsf{A}_N\right)
\left[\mathrm{Tr}^2\big(\mathsf{R}_N(z) \mathsf{B}_{N,\bullet}\big)-\mathrm{Tr}\big(\mathsf{R}_N(z) \mathsf{B}_{N,\bullet}\big)^2\right]\nonumber\\
&\hspace{20mm}=\frac{\alpha_N}{\lambda^2}\left[\mathrm{Tr}\big(\mathsf{P}_{N,\lambda}\mathsf{B}_{N,\bullet}\big)\mathrm{Tr}\big(\mathsf{C}_{N}\mathsf{B}_{N,\bullet}\big)-\mathrm{Tr}\big(\mathsf{P}_{\lambda_N}\mathsf{B}_{N,\bullet}\mathsf{C}_{N}\mathsf{B}_{N,\bullet}\big)\right].\label{Order2FWalks}
\end{align}
where 
\begin{equation}\label{CnDef}
\mathsf{C}_N:=\lim_{z\to1/\lambda^-}\big(\mathsf{I}-\mathsf{P}_{N,\lambda}\big)\mathsf{R}_N(z),
\end{equation}
and
$$
\alpha_N:=\lim_{z\to1/\lambda^-}(1-z\lambda)^{-1}\det\left(\mathsf{I}-z\mathsf{A}_N\right)=\prod_{\lambda_i\in\mathrm{Sp}(\mathsf{A}_N)\atop \lambda_i\neq \lambda_N}(1-\lambda_i/\lambda_N),
\vspace{-3mm}
$$
 which, by Lemma~\ref{RootedHikesLemma}, is asymptotically given by $\alpha^N$ for $N\gg 1$ and $\alpha=\tilde{\zeta}(1/\lambda)^{-1}$. 
The second line of Eq.~(\ref{Order2FWalks}) stems directly from the observation that $\mathsf{P}_{N,\lambda}$ is a rank one projector. Indeed for any such projector $\mathsf{Q}$ and matrix $\mathsf{M}$ of bounded norm we have $\mathrm{Tr}\big((\mathsf{QM})^2\big)=\mathrm{Tr}^2\big(\mathsf{QM}\big)$.

Each term of the second line of Eq.~(\ref{Order2FWalks}) can be evaluated generically. For convenience, let $c_0$ be the diagonal entry of $\mathsf{C}_G$, $c_1$ the entry relating first neighbours on $G$ and $c_{2,i}$, all the distinct entries of $\mathsf{C}_G$ relating second neighbours on $G$. Then $\mathrm{Tr}\big(\mathsf{P}_{N,\lambda}\mathsf{B}_{N,\bullet}\big)=2\lambda/N$, $\mathrm{Tr}\big(\mathsf{C}_{N}\mathsf{B}_{N,\bullet}\big)=2\lambda c_1$ and 
\begin{align*}
N\,\mathrm{Tr}\big(\mathsf{P}_{\lambda_N}\mathsf{B}_{N,\bullet}\mathsf{C}_{N}\mathsf{B}_{N,\bullet}\big)&=\sum_{\text{ordered}\atop\text{vertex pairs } v_i,\,v_j}\hspace{-3mm}\text{deg}(v_i)\text{deg}(v_j)\,\mathsf{C}_{v_iv_j},\\
&=(\lambda^2+\lambda)c_0+2\lambda^2c_1+\sum_{i} n_{2,i}\,c_{2,i},
\end{align*}
where $\text{deg}(v_i)$ is the degree of $v_i$ on the corolla $G_\bullet$ (see Fig.~\ref{FigCorolla} for an illustration of corolla graphs) and $n_{2,i}$ is the number of times entry $c_{2,i}$ appears. Recursion relations between the entries of $\mathsf{C}_G$ presented in Appendix~\ref{Recursion} give $c_0=0$, $c_1=-1$, $\sum_{i=1}^{\lfloor \lambda /2\rfloor} n_{2,i}c_{2,i}=-\lambda^2$.
\begin{figure}[!t]
\begin{center}
\vspace{-3mm}
\includegraphics[width=1\textwidth]{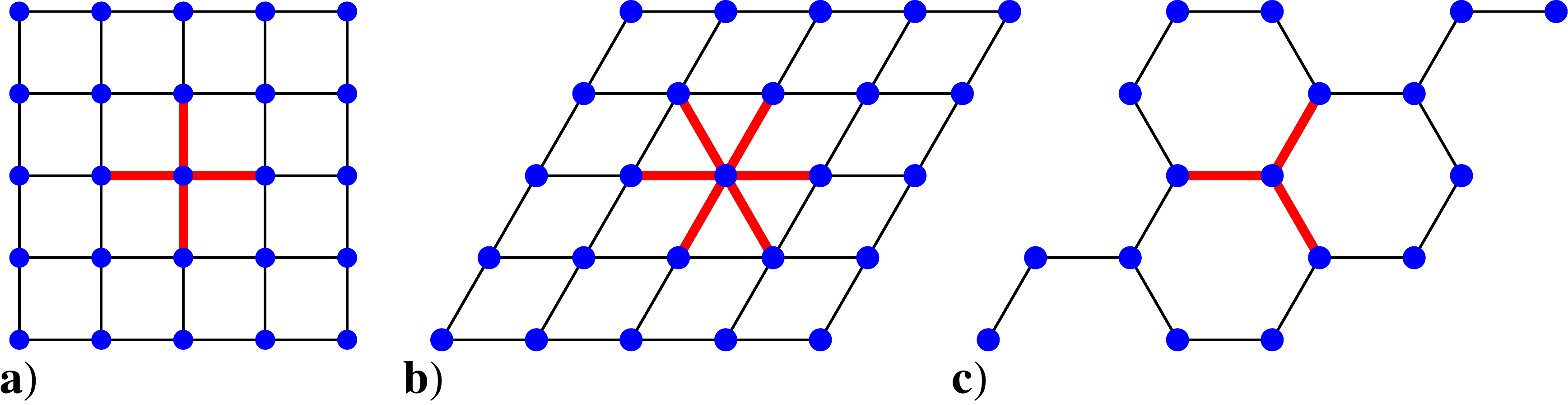}
\caption{In thick red edges, corolla graphs $G_\bullet$ on: \textbf{a)} the square lattice; \textbf{b)} the triangular lattice; \textbf{c)} the hexagonal lattice. The same construction and proof given in this section applies to any vertex transitive graph, in particular it is not limited to planar lattices and corollas can be considered in any dimension.}
\label{FigCorolla}
\end{center}
\vspace{-5mm}
\end{figure}
The second order  then generically evaluates to
$$
\frac{\alpha_N}{N\lambda^2}\left(-4\lambda^2-(-2\lambda^2-\lambda^2)\right)=-\frac{\alpha_N}{N}.
$$
We can finally put the zeroth, first and second orders of the determinant expansion together, yielding
$$
\det\left(\mathsf{I}-\frac{1}{\lambda}\mathsf{A}_{N}\right)=0+ 2\frac{\alpha_N}{N}-\frac{\alpha_N}{N}=\frac{\alpha_N}{N}.
$$
Asymptotically, for $N\gg1$, $\alpha_N\sim \alpha^N$ as per Lemma~\ref{RootedHikesLemma}, which gives the result. 
\end{proof}

The method employed in the proof of Lemma~\ref{WalkstoHikes} to calculate the fraction of hikes which are closed walks extends to any self-avoiding polygon. In this more general situation, the expansion of the determinant always terminates at a finite order which grows  with the polygon's length. 
The formulas to be evaluated become very involved however, to the point of being effectively impractical even for short SAPs. We nonetheless provide the general expression of the formal expansion at all orders for any SAP in Appendix~\ref{DetExpansion}.

To further exemplify this determinant expansion it is worth considering the calculation of the fraction $F_e$ of hikes which are walk whose last erased loop is an edge $e$ on the square lattice. In this case, the expansion of $F_e$ around $\det(\mathsf{I}-\frac{1}{\lambda} \mathsf{A}_{G})$ has exactly four non-zero orders each of which involves complicated sums of entries of the graph resolvent $\mathsf{R}(z)$
For example, the second order is asymptotically equal to
\begin{align*}
&\lim_{z\to1/\lambda^-}\frac{\alpha_N}{\lambda^2N}\Big(112\mathsf{R}_{10}(z)-38\mathsf{R}_{00}(z)-40\mathsf{R}_{11}(z)\\
&\hspace{30mm}-20\mathsf{R}_{20}(z)-12\mathsf{R}_{21}(z)-2\mathsf{R}_{30}(z)\Big)=\frac{\alpha^N}{\lambda^2 N}\times-10,
\end{align*}
for $N\gg1$. Although $\lambda=4$ on the square lattice, we left it unevaluated to help see that this is the second order of the expansion.
Here $\mathsf{R}_{ij}(z)$ designates the entry of the resolvent relating two vertices with distance $i$ along $x$ and $j$ along $y$.
Taken together, the four orders of the determinant expansion give asymptotically for $N\gg 1$
$$
F_e=\frac{1}{\lambda^2}\left(14\frac{\alpha^N}{\lambda N}-10\frac{\alpha^N}{\lambda^2 N}-80\frac{\alpha^N}{\lambda^3 N}+96\frac{\alpha^N}{\lambda^4 N}\right) = \frac{2\alpha^N}{N\lambda^2}= \frac{\alpha^N}{8N},
$$
where the overall $1/\lambda^2$ factor in front arises from the fact that the edge has length $2$ and would simply be $1/\lambda^{\ell(p)}$ for a general SAP.  
Lemma~\ref{WalkstoHikes} indicates that $\alpha^N/N$ relates hikes to closed walks, while the remaining $1/8$ factor is thus the asymptotic fraction of closed walks which are walk whose last erased loop is an edge cycle. This means in particular that, whenever $[z^\ell]R(z)\neq0$, 
$$
\lim_{\ell\to\infty} \frac{[z^\ell]R_e(z)}{[z^\ell]R(z)}=\frac{1}{8},
$$
where $R_e(z)$ is the generating function given in Eq.~(\ref{Redge}).\\

It might struck the reader that a remarkable number of simplifications must have taken place to yield such a simple answer as $1/8$ from the complicated orders of the expansion. This hints at the existence of a much simpler calculation procedure, and this is precisely the procedure in the spirit of Eq.~(\ref{Decompo1}). Now however, we will be able to remove the generic divergence of the naive approach, having separately obtained an asymptotic expression for the fraction of walks with respect to hikes by Lemma~\ref{WalkstoHikes}.

\begin{lemma}
Let $G$ be an infinite vertex-transitive graph of bounded-degree and let $\lambda$ be the supremum of its spectrum. Let $\{G^{\text{Tor}}_N\}_{N\in\mathbb{N}}$ be the small tori sequence of vertex-transitive graphs converging to $G$. Then the asymptotic fraction of closed walks which are walk whose last erased loop is $p$ is well defined and given by
\begin{align*}
\frac{F_p}{\lambda^{\ell(p)}}&=\lim_{N\to\infty}\,\,\lim_{z\to1/\lambda^{-}}\,\,z^{\ell(p)}\frac{\det\!\big(\mathsf{I}+z\mathsf{R}_N(z)\mathsf{B}_p\big)}{\det\!\big(\mathsf{I}+z\mathsf{R}_N(z)\mathsf{B}_\bullet\big)},\\
&=\frac{1}{\lambda^{\ell(p)+1}} \,\,\mathsf{deg}^{\mathrm{T}}.\,\mathrm{adj}\left(\mathsf{I}+\mathsf{C}_G\big|_p.\mathsf{B}_p\right)\,.\,\mathsf{1},
\end{align*} 
where $\mathsf{C}_G\big|_p$ is the restriction to $G_p$ of $\mathsf{C}_G:=\lim_{z\to1/\lambda^-}(\mathsf{I}-\mathsf{P}_\lambda)\mathsf{R}(z)$.
\end{lemma}

\begin{proof}
According to the Finite Sieve Theorem, the asymptotic fraction of hikes which are closed walks whose last erased loop is $p$ on $G^{\text{Tor}}_N$ is $\lim_{z\to 1/\lambda^-}z^{\ell(p)}\det(\mathsf{I}-z\mathsf{A}_{G^{\text{Tor}}_N\backslash p})$. Since all the limits taken here are finite and well defined (as everything takes place on $G^{\text{Tor}}_N$), the asymptotic fraction of closed walks whose last erased loop is $p$ on these finite graphs is 
\begin{align*}
\frac{\lim_{z\to 1/\lambda^-}z^{\ell(p)}\det(\mathsf{I}-z\mathsf{A}_{G^{\text{Tor}}_N\backslash p})}{\lim_{z\to 1/\lambda^-}z^0\det(\mathsf{I}-z\mathsf{A}_{G^{\text{Tor}}_N\backslash \bullet})}&=\lim_{z\to 1/\lambda^-}z^{\ell(p)}\frac{\zeta_N(z)\det(\mathsf{I}-z\mathsf{A}_{G^{\text{Tor}}_N\backslash p})}{\zeta_N(z)\det(\mathsf{I}-z\mathsf{A}_{G^{\text{Tor}}_N\backslash \bullet})},\\
&=\lim_{z\to 1/\lambda^-}z^{\ell(p)}\frac{\det\!\big(\mathsf{I}+z\mathsf{R}_N(z)\mathsf{B}_p\big)}{\det\!\big(\mathsf{I}+z\mathsf{R}_N(z)\mathsf{B}_\bullet\big)}.
\end{align*}
We now turn to studying the behaviour of the right hand side as $N\to\infty$.  To this end, we expand $\mathsf{I}+z\mathsf{R}_N(z)\mathsf{B}_p$ around $1/\lambda$ with $z<1/\lambda$. 

We need to distinguish behaviours based on the dimensionality $d>1$ of the lattice under study. We ignore the trivial 1D case (for which the only SAP is the edge, and the fraction of closed walks whose last erased loop is the left or right edge attached to any vertex is 1/2).
On $d>1$ dimensional lattices we have,
\begin{equation}\label{Eq:CLi}
\mathsf{I}+z\mathsf{R}_N(z)\mathsf{B}_p=\mathsf{C}_{N}\mathsf{B}_p-\frac{1}{\lambda \pi}\mathsf{P}_{N,\lambda}\mathsf{B}_p\,\text{Li}_{d/2}(1-z\lambda)+O(1-z\lambda),
\end{equation}
with $\mathsf{C}_{N}$ as defined in Eq.~(\ref{CnDef}) and $\text{Li}_a(x):=\sum_{n>1}\frac{x^n}{n^{a}}$ is the $a$-polylogarithm function. Combinatorially, it arises here from summations over closed walks weighted by $\lambda^{-\ell}$, which leaves a residual total weight asymptotically given by $\ell^{-d/2}$ for all closed walks of length $\ell\gg1 $. The generic nature of the behaviour exhibited by $\mathsf{I}+z\mathsf{R}_N(z)\mathsf{B}_p$ is now readily apparent: 1) divergence occurs only on 2D lattices, where it is logarithmic; 2) it is the same for all SAPs; and 3) it is also the same for all closed walks (which are readily recovered upon taking $p$ to be length 0, i.e. $\mathsf{B}_p\equiv\mathsf{B}_\bullet$ is a corolla). Thanks to these observations, the determinant expansion at $1/\lambda^-$ is
\begin{align*}
\det\big(\mathsf{I}+z\mathsf{R}_N(z)\mathsf{B}_p\big)&=\det\left(\mathsf{C}_{N}\mathsf{B}_p-\frac{1}{\lambda \pi}\mathsf{P}_{N,\lambda}\mathsf{B}_p\,\text{Li}_{d/2}(1-z\lambda)+O(1-z\lambda)\right),\\
&\hspace{-10mm}=-\frac{1}{\pi N}\text{Li}_{d/2}(1-z\lambda)\frac{1}{\lambda} \,\,\mathsf{deg}^{\mathrm{T}}.\,\mathrm{adj}\left(\mathsf{I}+\mathsf{C}_{N}.\mathsf{B}_p\right)\,.\,\mathsf{1}\\
&\hspace{55mm}+o\big(\text{Li}_{d/2}(1-z\lambda)\big),
\end{align*}
where we used the matrix-determinant lemma and the QR decomposition 
$$
-\frac{1}{\lambda \pi}\mathsf{P}_{N,\lambda}\mathsf{B}_p=-\frac{1}{ \pi N}\times \frac{1}{\lambda}\times\mathsf{1}.\mathsf{deg}^{\mathrm{T}}.
$$
This decomposition relies on the observation that $\mathsf{P}_{N,\lambda}$ is the projector onto $\mathsf{1}$, i.e. that all $G^{\text{Tor}}_N$ are regular.\footnote{This step can be adapted should we consider lattices which are not vertex-transitive but for which $G/\text{Aut}(G)$ has finitely many vertices. This is beyond the scope of this work.}
We recall that in the above expression, $\mathsf{deg}=\mathsf{B}_p.\mathsf{1}=\text{diag}(\mathsf{B}_p^2)$ is the degree of vertices on $G_p$. Similarly, at $1/\lambda^-$,
$$
\det\big(\mathsf{I}+z\mathsf{R}_N(z)\mathsf{B}_\bullet\big)=-\frac{1}{\pi N}\text{Li}_{d/2}(1-z\lambda)+o\big(\text{Li}_{d/2}(1-z\lambda)\big),
$$
and finally 
$$
z^{\ell(p)}\frac{\det\big(\mathsf{I}+z\mathsf{R}_N(z)\mathsf{B}_p\big)}{\det\big(\mathsf{I}+z\mathsf{R}_N(z)\mathsf{B}_\bullet\big)}=\frac{1}{\lambda^{\ell(p)+1}}\mathsf{deg}^{\mathrm{T}}.\,\mathrm{adj}\left(\mathsf{I}+\mathsf{C}_{N}.\mathsf{B}_p\right)\,.\,\mathsf{1}+o(1),
$$
which yields the result after taking the limits $z\to1/\lambda^-$ and $N\to\infty$ now both clearly well-defined, even when $d=2$. Combinatorially, the divergence curing on 2D lattices effected here comes from relating walk whose last erased loop is a fixed SAP to all closed walks rather than directly to the hikes. The relation between closed walks and hikes is performed separately through Lemma~\ref{WalkstoHikes}.\\

\begin{remark}
Recall that $R(z)$ and $R_p(z)$ are the ordinary generating functions of all closed walks on $G$ and of closed walk whose last erased loop is $p$ on $G$, respectively. 
The finite sieve here indicated that the limit in $z\to1/\lambda^-$ of the ratio of power series $R_p(z)/R(z)$ yields the asymptotic behaviour of the term-by-term ratio $[z^\ell]R_p(z)/[z^\ell]R(z)$ as $\ell\to\infty$, which rather corresponds to an Hadamard division of $R_p(z)$ by $R(z)$, i.e.
$$
\lim_{z\to1/\lambda^-}\frac{R_p(z)}{R(z)}=\lim_{\ell\to\infty }\frac{[z^\ell]R_p(z)}{[z^\ell]R(z)}.
$$
This is a corollary of the fact that the leading divergence of both $R_p(z)$
and $R(z)$ located at $z=1/\lambda$. For a general discussion on the relation between Hadamard products and singularity analysis, we refer to \cite{fill2005}. 
\end{remark}


%
\end{proof}

This concludes the proof of the Infinite Sieve Theorem. 
\end{proof}

\subsection{Probabilistic interpretations}
The results of the Infinite Sieve Theorem have a probabilistic interpretation which motivated the study of the fraction $F_p/\lambda^{\ell(p)}$ in the literature ultimately leading to SLE$_2$, albeit without the relation to hikes and the explicit form of Eq.~(\ref{FractionWalk}). Here the interpretation of $F_p/\lambda^{\ell(p)}$ as a probability distribution over the SAPs is first manifested in the following:
\begin{equation}\label{SumSAP}
\sum_{p:\,\text{SAP}}\frac{F_p}{\lambda^{\ell(p)}}=1.
\end{equation}
Combinatorially, this trivially states that any walk has a unique right prime factor \cite{SIAM2017} and consequently the total fraction of all closed walks whose last erased loop is \textit{any} SAP is exactly 1.  

\begin{remark}
All \emph{simple cycles} passing through some chosen vertex of $G$ are present in the sum of Eq.~(\ref{SumSAP}). Following the rules in the Cartier-Foata monoid of the hikes, a SAP of length $\ell(p)$ thus appears at least $2\ell(p)$ times in the sum. Here, the factor of 2 accounts for the orientation; and $\ell(p)$ reflects all the valid starting points for the SAP as a closed walk.\footnote{Rigorously, what matters is the size of the equivalence class of all words on edges that define the same hike $h$ in the original Cartier-Foata monoid. As mentioned earlier, this size is the value of the von Mangoldt function $\Lambda(h)$, which is the length of the unique right factor of a walk if $h$ is a walk and 0 otherwise \cite{SIAM2017}. Seen as a closed walk, a SAP $p$ is its own unique right-prime divisor and thus $\Lambda(p)=\ell(p)$.}  For example, on the square lattice, the $1\times 1$ square gets a factor of $8=2\times 4\times 1$ and the $1\times 2$ rectangle has a factor of $12=2\times 6$. Strictly speaking, the sum Eq.~(\ref{SumSAP}) therefore runs over simple cycles with fixed starting point and the SAP index is a (harmless) notational abuse. 
\end{remark}

We can go further in the probabilistic interpretation using purely combinatorial arguments:

\begin{proposition}\label{ProbaMeaning}
Let $G$ be an infinite vertex transitive planar graph and let $p$ be a SAP on it. Let $w$ be a random walk with uniform edge-transition probability $1/\lambda$. Run the walk until it comes back to its starting point. Then the probability $\mathbb{P}(w\mapsto p)$ that the last erased loop of $w$ be $p$ is equal to the fraction of all closed walks (including those passing an arbitrary number of times through the origin) whose last erased loop is $p$, $\mathbb{P}(w\mapsto p)=F_p/\lambda^{\ell(p)}$. 
\end{proposition}

\begin{proof}
Let $G_N$ be a family of finite vertex-transitive graphs of degree $\lambda$ converging to $G$ as $N\to\infty$. Let $W_{N, p}(z)$ be the ordinary generating function of closed walks on $G_N$ with right prime divisor $p$ and such that these walks never revisit their starting point except on their final step.  By construction, we have $\mathbb{P}(w\mapsto p)=\lim_{z\to1/\lambda^-}\lim_{N\to\infty} W_{N,p}(z)$ provided both limits are well defined. But 
\begin{align*}
W_{N, p}(z) &= z^{\ell(p)}\frac{\det\big(\mathsf{I}-z\mathsf{A}_{G_N\backslash p}\big)}{\det\big(\mathsf{I}-z\mathsf{A}_{G_N\backslash \bullet}\big)},\\
&=z^{\ell(p)}\Big(\zeta_N(z)\det\big(\mathsf{I}-z\mathsf{A}_{G_N\backslash p}\big)\Big)\Big(\zeta_N(z)\det\big(\mathsf{I}-z\mathsf{A}_{G_N\backslash \bullet}\big)\Big)^{-1}.
\end{align*}
This has a well defined limit when $N\to\infty$ as per the extension of Viennot's lemma to infinite graphs
\begin{align*}
W_p(z):=\lim_{N\to\infty}W_{N, p}(z) =z^{\ell(p)}\frac{\det\big(\mathsf{I}+z\mathsf{R}(z)\mathsf{B}_p\big)}{\det\big(\mathsf{I}+z\mathsf{R}(z)\mathsf{B}_\bullet\big)},
\end{align*}
and by the Infinite Sieve Theorem the limit $z\to1/\lambda^-$ of the above is also well defined. It is given by $\lim_{z\to1/\lambda^-}W_p(z)=F_p/\lambda^{\ell(p)}$ with $F_p$ as per Eq.~(\ref{FractionWalk}).
\end{proof}

The logarithm of $F_p$ also has a probabilistic interpretation that has not appeared in the literature so far. It is based on results by Espinasse and Rochet \cite{Espinasse2019}:

\begin{proposition}[Espinasse and Rochet, 2019]\label{EspinasseRochet}
Let $G$ be an infinite vertex-transitive graph and let $p$ be a SAP on it. Let $\mathbb{E}_w(.)$ designate the expectation value of a random variable with respect to the closed random walks defined up to translation and weighted with probability $\lambda^{-\ell(w)}$. 
Then
$$
\log(F_p) = \sum_{w:\,\text{walk}} \frac{\Lambda_p(w)}{\ell_p(w)}\lambda^{-\ell(w)} = \mathbb{E}_w\left(\frac{\Lambda_p(w)}{\ell_p(w)}\right).
$$
Here $\Lambda_p(w)$ counts the vertices that are both in $p$ and the unique right divisor of $w$ and $\ell_p(w)$ is the
number of vertices of $p$ visited by $w$, counted with multiplicity.
\end{proposition}

The reason for the notation $\Lambda_p(w)$ is that this function is similar to the von Mangoldt function $\Lambda$ defined on hikes \cite{SIAM2017}. Defining $\Lambda_p(h)$ to be 0 when the hike $h$ is not a walk and otherwise  $\Lambda_p(h)$ is as in the Proposition above, we have $\Lambda_p(h)=\Lambda(h)$ when $p$ is the unique right prime divisor of $h$.\\ 

The combination of Propositions~\ref{ProbaMeaning} and \ref{EspinasseRochet} implies the rather uncommon result that the entropy of the distribution of $F_p \lambda^{-\ell(p)}$ values itself has a probabilistic interpretation. Consequences of this observation
will be explored in a separate work.
%
%
~\\[-.5em]


\section{Discussion}\label{Discussion}
\subsection{Extension to SAWs and further lattices}
In this contribution, we presented fully deterministic combinatorial arguments based on number-theoretic sieves for counting walk multiples of SAPs on any finite or infinite vertex-transitive graphs. This is equivalent to counting all the walks whose last erased loop following Lawler's loop erasing procedure is some chosen self-avoiding polygon $p$ on such graphs. In fact, \text{all} the results presented here are immediately valid for self-avoiding walks (SAWs) as well. In particular, Eq.~(\ref{FractionWalk}) of the Infinite Sieve Theorem giving the fraction $F_p$ is immediately correct. For a self-avoiding walk $p$, $F_p/\lambda^{\ell(p)}$ is the fraction of all open walks with the same fixed starting and end points for which $p$ is the self-avoiding skeleton remaining after loop-erasing. 

Finally, the arguments presented here should extend without fundamental changes to infinite graphs that are not vertex transitive as long as $G/\text{Aut}(G)$ is finite. This nonetheless requires further work and is beyond the scope of the present contribution.

\subsection{Counting the self-avoiding polygons}\label{UnivDiscussion}
We recall that $R(z)$ and $R_p(z)$ are the ordinary generating functions of closed walks and of walks whose last erased loop is a SAP $p$, respectively. 

The research presented here suggests a natural strategy to tackle the open problem of asymptotically counting SAWs and SAPs. First, observe that we know the exact number $R(z)[L]$ of closed walks of length $L$ defined up to translation. Then, if we could determine the exact number $R_p(z)[L]$  of closed walks of length exactly $L$ whose last erased loop is $p$, it would be sufficient to sum this over all SAPs of length strictly less than $L$ and subtract the result from $R(z)[L]$ to determine the number $\pi(L)$ of SAPs of length exactly $L$: 
\begin{equation*}
\pi(L)=R(z)[L]-\sum_{p:\,\text{SAP}\atop \ell(p)<L}R_p(z)[L].
\end{equation*}
While such a precise count is not feasible in practice as $L\to\infty$, an asymptotic estimate of the number of walks whose last erased loop is any chosen SAP may seem, at first, to be sufficient to gain an insight into the number of SAPs themselves. Following this idea, we would rather write
\begin{equation}\label{PiEstimate}
\frac{\pi(L)}{R(z)[L]}=1-\sum_{p:\,\text{SAP}\atop \ell(p)<L}\frac{R_p(z)[L]}{R(z)[L]}.
\end{equation}
and use $R_p(z)[L]/R(z)[L]\sim F_p \lambda^{-\ell(p)}$ for $L\gg 1$. Thus, we would only need to estimate sums like
$$
S(L):=\sum_{p:\text{ SAP}\atop \ell(p) \leq L}\frac{F_p}{\lambda^{\ell(p)}},
$$
for $L\gg 1$, in order to work out an asymptotic expansion for $\pi(L)$. Such an estimate can already be determined from R. Kenyon's seminal results \cite{Kenyon2000}, we find
\begin{equation}\label{LSL}
S(L)=1-L^{-3/5}+O(L^{-3/5}).
\end{equation}
\begin{figure}[t!]
\includegraphics[width=1\textwidth]{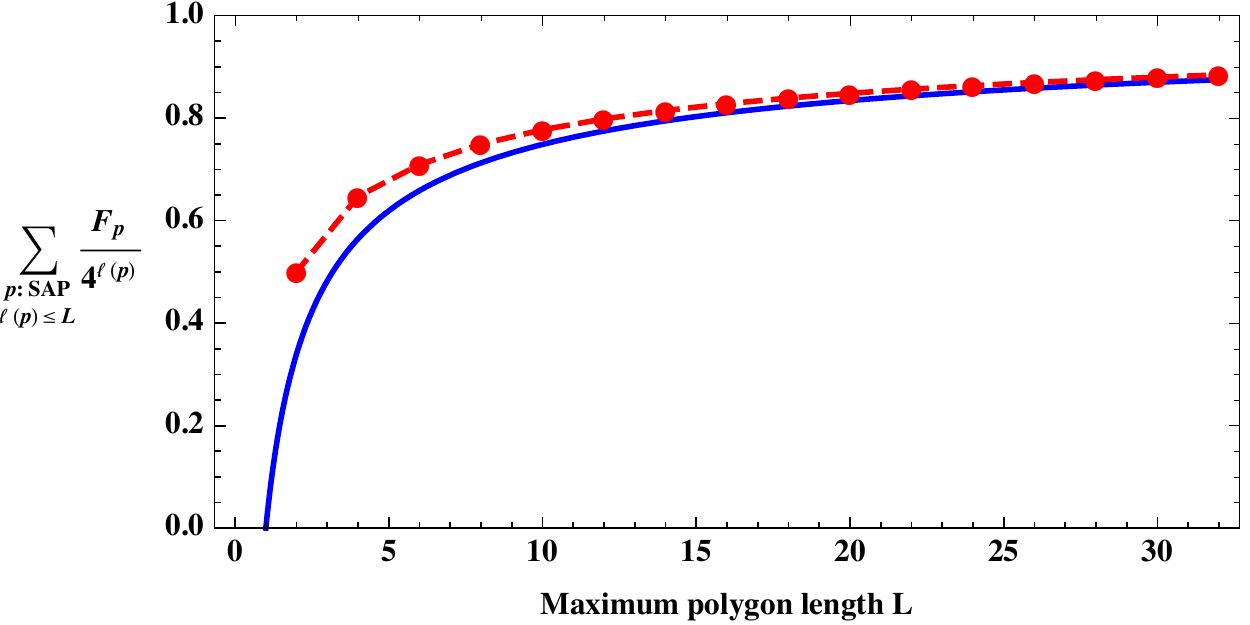}
 \caption{\footnotesize \label{FSeries} In red points and dashed red line, the sum over all self-avoiding polygons with fixed starting point on the square lattice of length at most $L$ of the fraction $F_p/4^{\ell(p)}$, as a function of $L$. The corresponding data table is presented in Appendix~\ref{Data}. The solid blue line is the first term of asymptotic expansion for this quantity, that is $1-L^{-3/5}$. The discrepancy between $-3/5$ and the $-1/2$ expected from SAP counting is due to an accumulation of error terms in the Infinite Sieve Theorem, which require us to include terms beyond the dominant $F_p/4^{\ell(p)}$ or use stronger two-sided sieves.}
\end{figure}
See also Fig~\ref{FSeries} for a numerical illustration.
This result of course wildly differs from the $(\mu/\lambda)^L L^{-1/2}$ expected here from the numerically conjectured scaling for $\pi(L)$.\footnote{The correction term is $L^{-1/2}$ here because we count simple cycles rather than SAPs. This is responsible for a factor of $L$ in front of the $L^{-5/2}$. Since in addition, $R(z)[L]\sim \lambda^L/(\pi L)$ for $L\gg1$, this accounts for another factor of $L$ and finally we get $L\times L\times L^{-5/2}=L^{-1/2}$.}
From the point of view of probability theory, the origin of this discrepancy is clear: the law governing Lawler's loop erased random walks essentially converges to $\text{SLE}_2$ rather than the conjectured $\text{SLE}_{8/3}$ for SAP and SAW models. From the point of view of sieve techniques however, the chasm between these results originates from an uncontrolled accumulation of error terms affecting the estimate $R_p(z)[L]/R(z)[L]\sim F_p \lambda^{-\ell(p)}$.\footnote{These errors have the same origin as those affecting the  Eratosthenes-Legendre sieve in number theory!} 

It is important to recall that $F_p \lambda^{-\ell(p)}$ is only the first, asymptotically dominant term of the asymptotic expansion of the number of walks whose last erased loop is $p$. In particular $F_p \lambda^{-\ell(p)}$ is a good approximation to $R_p(z)[L]/R(z)[L]$ only when $L\gg \ell(p)$ (see Appendix~\ref{ErrorTerms}). Yet, when we subtract walks with fixed last erased loops from all closed walks of length $L$, we must consider the walks whose last erased loop $p$ is of length up to $\ell(p)=L-1$. 
Given the exponential growth in the number of SAPs, this means that \emph{most} of our estimates are affected by large, uncontrolled errors, and it is impossible to exploit Eq.~(\ref{PiEstimate}) using solely $S(L)$.\\[-.5em]

This problem has two potential solutions. The first idea is  to take into account some error terms $\text{Err}_p(L)$  in the asymptotic expansion of $R_p(z)[L]/R(z)[L]$ so as to determine this quantity more precisely. Since all the error terms are exactly available\footnote{In fact $R_p(z)[L]$ itself is in principle exactly available from the extention of Viennot's Lemma to infinite graphs. In this situation however, it is a precise estimate for the sum over SAPs of $R_p(z)[L]$ which is utterly lacking.}, it seems possible that an extension to Kenyon's arguments would allow us to estimate sums of such error terms generically; just as Eq.~(\ref{LSL}) does for the dominant term. This idea suffers from a major drawback: error terms actually \emph{grow} with $L$ if we consider classes of SAPs for which $L-\ell(p)$ is fixed. Since most SAPs of length up to $L-1$ are close in length to $L$, this means that the overall error term affecting Eq.~(\ref{PiEstimate}) grows uncontrollably with $L$. Thus, an increasingly (and unrealistically) detailed knowledge of the errors is needed as $L\to\infty$, so that this strategy collapses completely with respect to rigorous arguments. 

\subsection{The path to rigorous progress}
The second approach relies on a crucial foundational work by M. Bousquet-M\'elou regarding the enumeration of heaps of pieces satisfying both left and right constraints \cite{bousquet-melou1993,bousquet-melou1991}. This work opens the way for two-sided sieves in the same manner as Viennot's Lemma relates to the Finite and Infinite Sieve Theorems: they give control over both the left and right prime divisors of a walk. Consequently, the maximum length of the primes to be considered in Eq.~(\ref{PiEstimate}) is reduced to only $L/2$. The ``sieving gap'' between $L$ and $L/2$ dramatically reduces the importance of the error terms to the extend that, in accordance with Appendix~\ref{ErrorTerms}, we expect them all to vanish under the limit $L\to\infty$.\footnote{Seeing heaps of pieces as an extension of number theory as in \cite{SIAM2017}, shows that this $L/2$ is the extension of the $\sqrt{x}$ gap present in all standard number-theoretic sieves. We can similarly show that the fraction $F_p \lambda^{-\ell(p)}$ extends the quantity $\log(x)/\log(p)$ and all identities given here extend (and hence reduce to) valid number-theoretic identities. Non-trivial (novel) results on partial sums of the M\"obius function also follow heuristically.} At the same time, the dominant contribution is not $F_p/\lambda^{\ell(p)}$ anymore. The resulting calculations will be presented in a separate work.~\\

\begin{appendices}
\section{Converging sequences of graphs}\label{GraphSeq}
We here recall the notion of a converging sequence of graphs. We follow directly the work of \cite{Benjamini2001}:\\[-.5em] 

Let $G$ and $G'$ be two bounded degree graphs and let $r$ and $r'$ be vertices called the roots on $G$ and $G'$, respectively. A topology on the space $\mathcal{X}$ of isomorphism classes of rooted connected graphs is induced by the following metric. Let $B_G(r,n)$, $n\in\mathbb{N}^*$ be the ball of radius $n$ centred on the root $r$ in graph $G$. Let $k$ be the supremum of all $n$ such that  $(B_G(r,n),r)$ and $(B_{G'}(r',n),r')$ be isomorphic as rooted graphs and define the distance $d$ between $(G,r)$ and $(G',r')$ as $2^{-k}$. Then $d$ is a metric on $\mathcal{X}$. 
Now we say that a sequence $(G_N, r_N)$ of rooted graphs converges to a rooted graph $(G,r)$ if and only if $\lim_{N\to\infty}d\big((G_N, r_N), (G, r)\big)=2^{-\infty}:=0$.\\[-.5em] 

In this work we specifically deal with two sequences converging to infinite vertex-transitive graphs (termed \emph{lattices}):
\begin{itemize}
\item[i)] Small graph sequence: where $G_{n}$ is a finite cut-out of the lattice $G$ that includes $B_G(r,n)$ but not $B_G(r,n+1)$.
\end{itemize}
In the context of planar lattices we construct a specific small graph sequence as follows. Define $G^{\text{Sq}}_{N}$ the induced subgraph of $G$ which is the cut-out of $G$ contained within the square of side length $n\in\mathbb{N}^*$ centred on $r$ (edges being given a length of 1). Here $N$ stands for the number of vertices of this square cut-out, the precise relation between $n$ and $N$ depending on the underlying lattice $G$. We can now also define: 
\begin{itemize}
\item[ii)] Small torus sequence: where the small graphs $G^{\text{Sq}}_{N}$ are wrapped around a torus. The resulting graphs $G^{\text{Tor}}_N$ are vertex-transitive tori with the same degree as $G$.
\end{itemize}
Since all tori $G^{\text{Tor}}_N$ as well as $G$ are vertex-transitive, the roots $r_N$ and $r$ are irrelevant when considering convergence of this sequence. We therefore simply say that the sequence of graphs $G^{\text{Tor}}_N$ converges to the lattice $G$.\\[-.5em]

\begin{figure}
\begin{center}
\includegraphics[width=1\textwidth]{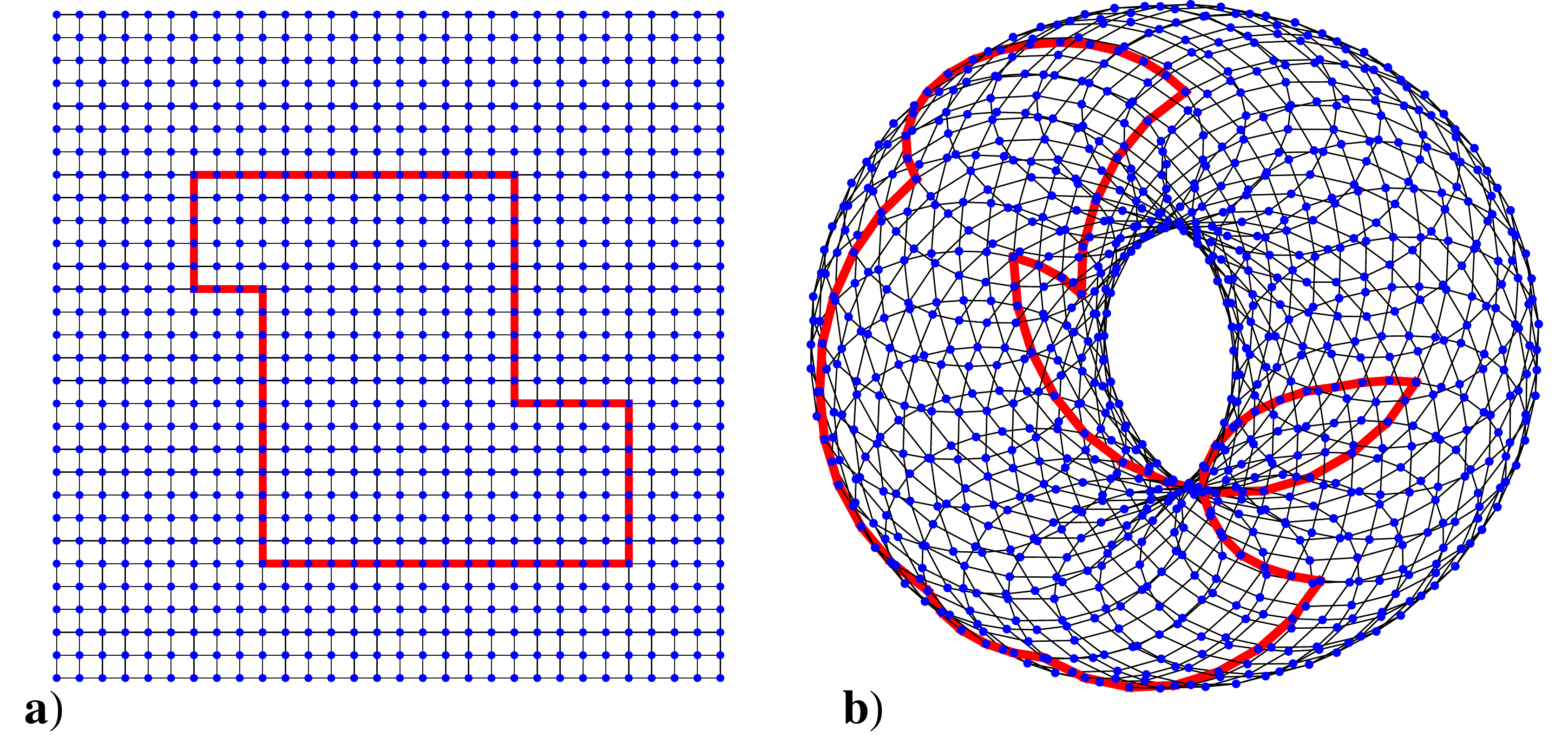}
  \end{center}
  \vspace{-4.5mm}
 \caption{\label{GraphSeqFIG}{\footnotesize Examples of finite graphs part of the sequences converging to the (infinite) square lattice: \textbf{a)} small graph $G^{\text{Sq}}_{30}$ with a self-avoiding polygon on it; \textbf{b)} corresponding small torus $G^{\text{Tor}}_{30}$ with the same self-avoiding polygon.}}
  \vspace{-2mm}
\end{figure}
The small torus sequence generalises to non-planar lattices by wrapping around genus $g>2$ tori the small graphs $G^{\text{Hyp}}_{N}$ defined as the induced subgraph of $G$ which is the cut-out of $G$ contained within the $d$-hypercube of side length $n\in\mathbb{N}^*$ centred on $r$. In this case, the dimension $d>2$ is the smallest dimension of the space $\mathbb{R}^d$ such that $G$ can be embedded in $\mathbb{R}^d$ in a way that no two edges cross.

\section{Viennot's lemma on infinite graphs with formal series}\label{ViennotFormal}
The proof is entirely similar to that given in the case of ordinary generating functions Proposition~\ref{InfiniteViennot}. We have:
\begin{proposition}\label{InfiniteViennotFormal}
Let $G$ be an infinite (di)graph with bounded degree with labelled adjacency matrix $\mathsf{W}$. Let $p$ be a prime on $G$, with support $\mathcal{V}(p)$ and neighborhood $\mathcal{N}(p)$. Let $G_p\prec G$ be the induced subgraph of $G$ with vertex set $\mathcal{V}(G_p)=\mathcal{V}(p)\cup\mathcal{N}(p)$ and $\mathsf{W}_p$ its labelled adjacency matrix. Let $\mathsf{R}=\big(\mathsf{I}-\mathsf{W}\big)^{-1}\big|_{G_p}$ be the restriction of the formal resolvent of $G$ to $G_p$.  

Then the formal series of closed walks whose last erased loop is $p$ is given by
\begin{equation}\label{ViennotLemmaInfiniteFormal}
\sum_{w~\text{walk}\atop p|_r w} w=\det\!\big(\mathsf{I}+\mathsf{R}\,\mathsf{W}_p\big)p.
\end{equation}
\end{proposition}

\section{Recursion relations for $\mathsf{C}$ and $\mathsf{R}(z)$}\label{Recursion}
Recursion relations between entries of the resolvent naturally arise from the basic observations that walks of length $\ell$ going from vertex $v_i$ to vertex $v_j$ were walks of length $\ell-1$ from $v_i$ to a neighbours of $v_j$. We have a similar relation relating walks of length $\ell$ with those of length $\ell-2$. In terms of generating functions, these read
\begin{align*}
&\lambda z\mathsf{R}(z)_{1}=R(z)-1,\\
&z^2\sum_{i} n_{2,i} \mathsf{R}(z)_{2,i}+\lambda z^2 R(z)=R(z)-1
\end{align*}
where $R(z)$, $\mathsf{R}(z)_{1}$ and $\mathsf{R}(z)_{2,i}$ designate a diagonal entry of the resolvent, the entry of the resolvent relating first neighbours, and entries of the resolvent relating second neighbours, respectively. The quantity $n_{2,i}$ is the number of times $\mathsf{R}(z)_{2,i}$ appears. Similar recursion relations are thus obeyed by the entries of $\mathsf{C}$
$$
c_1=c_0-1,~~\frac{1}{\lambda} c_0+\sum_{i=1}^{\lfloor \lambda /2\rfloor} n_{2,i}c_{2,i}=\lambda^2(c_0-1).
$$
where $c_0$, $c_1$ and $c_{2,i}$ are defined similarly to $R(z)$, $\mathsf{R}(z)_{1}$ and $\mathsf{R}(z)_{2,i}$ respectively.\\ 

These recursion relations simplify further on noting that $c_0=0$:
\begin{lemma}
Let $G$ be an infinite vertex transitive graph with  bounded degree $\lambda$ and resolvent $\mathsf{R}(z)$. Let $\mathsf{P}_\lambda$ be the projector onto the eigenspace associated with eigenvalue $\lambda$ and let
$\mathsf{C}_G:=\lim_{1/\lambda^-}\big(\mathsf{I}-\mathsf{P}_\lambda\big)\mathsf{R}(z)$. Then
all diagonal entries $c_0$ of $\mathsf{C}_G$ are 0. 
\end{lemma}
\begin{proof}
Since $G$ is regular $(\mathsf{I}-\mathsf{P}_\lambda\big)_{ij}=(1-1/N)\delta_{i,j}+(1/N)\delta_{i\neq j}$. Then
$$
\Big(\big(\mathsf{I}-\mathsf{P}_\lambda\big)\mathsf{R}(z)\Big)_{kk}=\left(1-\frac{1}{N}\right)\mathsf{R}(z)_{k,k}+\frac{1}{N}\sum_{j\neq k} \mathsf{R}(z)_{j,k}=0.
$$
\end{proof}

\section{All orders of the determinant expansion}\label{DetExpansion}
The expansion of determinants such as $\det(\mathsf{M}_1+\epsilon\mathsf{M}_2)$ around $\det(\mathsf{M}_1)$ is well known. Here we report the slight variant valid when $\mathsf{M}_1$ is singular. Indeed in the case of interest here $\mathsf{M}_1\equiv \mathsf{I}-z\mathsf{A}_G$ which is singular in $z=1/\lambda$. Taking into consideration that the adjugate matrix of a singular matrix with a kernel of dimension 1 is the projector  $\mathsf{P}_{\lambda}$ onto that kernel times a constant which, for us, is asymptotically $\alpha^N$, $N\gg1$, the $k$th order of the expansion of $\det(\mathsf{I}-\lambda^{-1}\mathsf{A}_{G\backslash p})$ for an arbitrary SAP $p$ is:
\begin{align*}
&\text{Order}^{(k)}=\frac{\alpha^N}{\lambda^k}\sum_{\pi\in\Pi(k)}(-1)^{|\pi|}\sum_{j=1}^k\Bigg\{\frac{1}{(\pi_j-1)!j^{\pi_j-1}}\mathrm{Tr}\left[\big(-z \mathsf{P}_{\lambda}\mathsf{B}_p\big)\big(-z \mathsf{C}\mathsf{B}_p\big)^{j-1}\right]\\
&\hspace{27mm}\times\mathrm{Tr}\left[\big(-z \mathsf{C}_G|_p\mathsf{B}_p\big)^{j}\right]^{\pi_j-1}\prod_{m=1\atop m\neq j}^k\mathrm{Tr}\left[\big(-z \mathsf{C}_G|_p\mathsf{B}_p\big)^{m}\right]^{\pi_m}\frac{1}{m!\,m^{\pi_m}}\Bigg\}.
\end{align*}
In this expression, $\Pi(k)$ designates the set of partitions $\pi:=\{\pi_i\}_{1\leq i\leq k}$ of $k$ such that $\sum_i i\pi_i =k$ and $|\pi|=\sum_i \pi_i$. The matrices $\mathsf{C}_G|_p$ and $\mathsf{B}_p$ are as defined in the Infinite Sieve Theorem. 

\section{Error terms in the Infinite Sieve Theorem}\label{ErrorTerms}
The Infinite Sieve Theorem establishes the dominant term in the asymptotic expansion of the fraction of closed walks whose unique right prime factor is a given SAP $p$. This term is dominant only under the limit $\ell\to\infty$, where the closed walks under consideration have diverging length. In this section we evaluate the accompanying error terms which come into consideration for finite length walks. Concretely, these error terms cannot be neglected when this finite length is `too close' to that of the SAP $p$, in a sense which will be made precise below. 

%

The Infinite Sieve Theorem produces the asymptotic expansion of the extension of Viennot's Lemma to infinite graphs which gives, for the generating function $R_p(z)$ of walks whose last erased loop is the SAP $p$,
$$
R_p(z)=z^{\ell(p)}\det\big(\mathsf{I}+z\mathsf{R}(z)\mathsf{B}_p\big).
$$
The various divergences encountered in the asymptotic expansion were specifically avoided with the following form,
\begin{align}\label{RpAppendix}
R_p(z)&=R(z)\left(\frac{z^{\ell(p)}}{R(z)}\det\big(\mathsf{I}+z\mathsf{R}(z)\mathsf{B}_p\big)\right).
\end{align}
This presentation of $R_p(z)$ is also supported by its combinatorial meaning: with Lemma~\ref{WalkstoHikes} the first term relates to the density of walks amongst the hikes, while the terms in parentheses give rise to the finite asymptotic fraction of walks whose last erased loop is $p$ among all closed walks. 

As we have seen, we should not expect any of the generating functions appearing here to be meromorphic in general. For this reason, we should only use properties of products of generating functions rather than ratios. One such standard property is recalled below:

\begin{lemma}\label{FGLemma}
Let $F(z)=\sum_{\ell\geq 0}\lambda^\ell f(\ell) z^{\ell}$ with $\lambda$ real and $f(\ell)$ a function such that $\lim_{\ell\to\infty}f(\ell)$ exists and is finite. Then for any function $G(z)$ of $z$ such that $G(1/\lambda)$ is finite,
$$
\big(F(z) G(z)\big)[L]=\left(G(1/\lambda)+\sum_{k>0} \frac{\nabla^{k}[f](L)}{\lambda^k\,k!\,f(L)}G^{(k)}(1/\lambda)\right)\times F(z)[L],
$$
where $G^{(k)}(1/\lambda)$ stands for the $k$th derivative of $G(z)$ evaluated in $1/\lambda$.
\end{lemma}

Returning to Eq.~(\ref{RpAppendix}), we introduce the walk correction function $f_w(\ell):=R(z/\lambda)$, so that the number of closed walks of length $\ell$ on $G$ is exactly $\lambda^\ell f_w(\ell)$. On the square lattice, we have 
$$
f_w(2\ell)=\binom{2\ell}{\ell}^24^{-2\ell}\sim\frac{1}{\ell\pi}~\text{ for }\ell\gg1,
$$
while $f_w(2\ell+1)=0$, and $\lim_{\ell\to\infty }f_w(\ell)=0$. In general, on a $d$ dimensional lattice $f_w(\ell)$ is on the order of $\ell^{-d/2}$ for $\ell\gg 1$.

We may now use Lemma~\ref{FGLemma} on Eq.~(\ref{RpAppendix}). Together with the Infinite Sieve Theorem, this gives 
$$
[z^\ell]R_p(z)=\left(\frac{F_p}{\lambda^{\ell(p)}}+\sum_{k\geq 1}^\infty\frac{\nabla^k[f_w](\ell-\ell(p))}{f_w(\ell)\lambda^k\,k!}\,\text{Frac}_p^{(k)}(1/\lambda)\right)[z^\ell]R(z),
$$
with $\text{Frac}_p(z):=R_p(z)/R(z)$ and $F_p$ as per Eq.~(\ref{FractionWalk}). Let us denote, 
$$
\text{Err}_p(2\ell):=\sum_{k\geq 1}^\infty\frac{\nabla^k[f_w](2\ell-\ell(p))}{f_w(2\ell)\lambda^k\,k!}\,\text{Frac}_p^{(k)}(1/\lambda),
$$
and observe that since $\text{Frac}_p^{(k)}(1/\lambda)$ does not depend on $\ell$, we may now estimate the decay of the correction terms with respect to the dominant  contribution $F_p/\lambda^{\ell(p)}$ as $\ell\to\infty$. Taking $f_w(\ell)\sim \ell^{-d/2}$ as a guide on a $d$ dimensional lattice, two situations arise:
\begin{itemize}
\item[i) ] If $\ell\gg1 $ and $\ell-\ell(p)\gg 1$, then the error terms decrease proportionally to $\big(\ell-\ell(p)\big)^{-1}$ on all $d$-dimensional vertex transitive lattices; 
\item[ii)] If $\ell\gg 1$ but $\ell-\ell(p)$ is on the order of 1, then the error terms \emph{actually increase} with $\ell$, as $\ell^{d/2}$.
\end{itemize}
This analysis indicates that $[z^\ell]R(z)\times F_p/\lambda^{\ell(p)}$ is a good approximation to $[z^\ell]R_p(z)$ when $\ell$ is much larger than $\ell(p)$. This confirms the discussion of \S\ref{Discussion}: using $F_p/\lambda^{\ell(p)}$ to obtain an estimate $\pi(L)$ via Eq.~(\ref{PiEstimate}) is misguided precisely because most of the SAPs fall in situation ii), where using $F_p/\lambda^{\ell(p)}$ to estimate $[z^\ell]R_p(z)$ is outright wrong.

\section{Data table}\label{Data}
Table~\ref{Table:data} gives all computed values of 
$$ 
S(L):=\sum_{p:\,\text{SAP}\atop \ell(p)\leq L} \frac{F_p}{4^{\ell(p)}}
$$
as a function of $L$ and evaluated on the square lattice. Although most tabulated values of $S(L)$ were \emph{computed analytically}, we here report only numerical results rounded at $10^{-4}$ owing to length concerns. At length 32, $S(L)$ requires computing $F_p$ values for $3,484,564,613$ self-avoiding polygons.

\begin{table}
\centering
\renewcommand{\arraystretch}{1.2}
\begin{tabular}{|c|c|c|c|c|c|c|c|}
\hline
\hline
\bf Length $L$ &2*&4*&6*&8*&10*&12*&14*\\
\hline 
  $S(L)$  &0.5& 0.6473& 0.7093& 0.7493& 0.7774& 0.7984& 0.8149\\
\hline
\bf Length $L$ &16*&18*&20*&22*&24*&26*&28*\\
\hline
 $S(L)$ &0.8282& 0.8392& 0.8485& 0.8565& 0.8635& 0.8696& 0.8751\\
 \hline
\bf Length $L$ &30&32&&&&&\\
\hline
 $S(L)$& 0.8799& 0.8843&&&&&\\
\hline
\hline
\end{tabular}
\caption{\label{Table:data}Table of value for the sum $S(L)$ as a function of $L$. The asterisk $\ast$ indicates that all calculations for this length were analytical.}
\end{table}

\end{appendices}

\bibliographystyle{elsarticle-num}

\begin{thebibliography}{10}
\expandafter\ifx\csname url\endcsname\relax
  \def\url#1{\texttt{#1}}\fi
\expandafter\ifx\csname urlprefix\endcsname\relax\def\urlprefix{URL }\fi
\expandafter\ifx\csname href\endcsname\relax
  \def\href#1#2{#2} \def\path#1{#1}\fi

\bibitem{Smirnov2010}
S.~Smirnov, {Conformal invariance in random cluster models. I. Holomorphic
  fermions in the Ising model}, Annals of Mathematics 172 (2010) 1435--1467.
\newblock \href {http://dx.doi.org/10.4007/annals.2010.172.1435}
  {\path{doi:10.4007/annals.2010.172.1435}}.

\bibitem{Werner2000}
W.~Werner, \href{https://doi.org/10.1090/S0894-0347-07-00557-7}{The conformally
  invariant measure on self-avoiding loops}, Journal of the American
  Mathematical Society 21 (2000) 137--169.
\newblock \href {http://dx.doi.org/10.1090/S0894-0347-07-00557-7}
  {\path{doi:10.1090/S0894-0347-07-00557-7}}.
\newline\urlprefix\url{https://doi.org/10.1090/S0894-0347-07-00557-7}

\bibitem{Lawler2000}
G.~Lawler, W.~Wendelin, Universality for conformally invariant intersection
  exponents, Journal of the European Mathematical Society 2 (2000) 291--328.
\newblock \href {http://dx.doi.org/10.1007/s100970000024}
  {\path{doi:10.1007/s100970000024}}.

\bibitem{beffara2013}
V.~Beffara, H.~Duminil-Copin, \href{https://doi.org/10.1214/11-PS186}{{Planar
  percolation with a glimpse of Schramm-Loewner evolution}}, Probability
  Surveys 10 (2013) 1--50.
\newblock \href {http://dx.doi.org/10.1214/11-PS186}
  {\path{doi:10.1214/11-PS186}}.
\newline\urlprefix\url{https://doi.org/10.1214/11-PS186}

\bibitem{Mandelbrot1982}
B.~Mandelbrot, \href{https://cds.cern.ch/record/98509}{{The fractal geometry of
  nature}}, Freeman, San Francisco, CA, 1982.
\newline\urlprefix\url{https://cds.cern.ch/record/98509}

\bibitem{Lawler2001}
G.~Lawler, O.~Schramm, W.~Werner,
  \href{http://dx.doi.org/10.4310/MRL.2001.v8.n4.a1}{The dimension of the
  planar brownian frontier is 4/3}, Mathematical Research Letters 8 (2001)
  401--411.
\newblock \href {http://dx.doi.org/10.4310/MRL.2001.v8.n4.a1}
  {\path{doi:10.4310/MRL.2001.v8.n4.a1}}.
\newline\urlprefix\url{http://dx.doi.org/10.4310/MRL.2001.v8.n4.a1}

\bibitem{Lawler2004}
G.~Lawler, O.~Schramm, W.~Werner, On the scaling limit of planar self-avoiding
  walk, in: Proceedings of Symposia in Pure Mathematics, Vol.~72, 2004, pp.
  339--364.

\bibitem{beffara2008}
V.~Beffara, \href{https://doi.org/10.1214/07-AOP364}{{The dimension of the SLE
  curves}}, The Annals of Probability 36~(4) (2008) 1421--1452.
\newblock \href {http://dx.doi.org/10.1214/07-AOP364}
  {\path{doi:10.1214/07-AOP364}}.
\newline\urlprefix\url{https://doi.org/10.1214/07-AOP364}

\bibitem{Copin2012}
H.~Duminil-Copin, S.~Smirnov, \href{http://www.jstor.org/stable/23234646}{The
  connective constant of the honeycomb lattice equals $\sqrt{2+\sqrt{2}}$},
  Annals of Mathematics 175~(3) (2012) 1653--1665.
\newline\urlprefix\url{http://www.jstor.org/stable/23234646}

\bibitem{Lawler1980}
G.~Lawler, \href{https://doi.org/10.1215/S0012-7094-80-04741-9}{A self-avoiding
  random walk}, Duke Mathematical Journal 47~(3) (1980) 655--693.
\newblock \href {http://dx.doi.org/10.1215/S0012-7094-80-04741-9}
  {\path{doi:10.1215/S0012-7094-80-04741-9}}.
\newline\urlprefix\url{https://doi.org/10.1215/S0012-7094-80-04741-9}

\bibitem{Lawler20042}
G.~Lawler, W.~Werner, \href{https://doi.org/10.1007/s00440-003-0319-6}{{The
  Brownian loop soup}}, Probability Theory and Related Fields 128~(4) (2004)
  565--588.
\newblock \href {http://dx.doi.org/10.1007/s00440-003-0319-6}
  {\path{doi:10.1007/s00440-003-0319-6}}.
\newline\urlprefix\url{https://doi.org/10.1007/s00440-003-0319-6}

\bibitem{Lawler2007}
G.~Lawler, J.~Trujillo~Ferreras,
  \href{https://doi.org/10.1090/S0002-9947-06-03916-X}{Random walk loop soup},
  Transactions of the American Mathematical Society 359 (2007) 767--787.
\newblock \href {http://dx.doi.org/10.1090/S0002-9947-06-03916-X}
  {\path{doi:10.1090/S0002-9947-06-03916-X}}.
\newline\urlprefix\url{https://doi.org/10.1090/S0002-9947-06-03916-X}

\bibitem{Lawler2018}
G.~Lawler, Conformally invariant loop measures, in: Proceedings of the
  International Congress Of Mathematicians, Rio de Janeiro, Vol.~1, 2018, pp.
  669--704.

\bibitem{cartier1969}
P.~Cartier, D.~Foata, Probl\`emes combinatoires de commutation et
  r\'earrangements, Vol.~85 of Lecture Notes in Mathematics, Springer-Verlag
  Berlin Heidelberg, 1969.

\bibitem{Viennot1986}
G.~X. Viennot, Heaps of pieces, i : Basic definitions and combinatorial lemmas,
  in: G.~Labelle, P.~Leroux (Eds.), Combinatoire {\'e}num{\'e}rative, Springer
  Berlin Heidelberg, Berlin, Heidelberg, 1986, pp. 321--350.

\bibitem{Krattenthaler2006}
C.~Krattenthaler, {The theory of heaps and the Cartier-Foata monoid},
  complement to the electronic republication to Cartier's and Foata's
  \textit{Probl\`emes combinatoires de commutation et r\'earrangements}.
  (2006).

\bibitem{SIAM2017}
P.-L. Giscard, P.~Rochet, \href{https://doi.org/10.1137/15M1054535}{Algebraic
  combinatorics on trace monoids: Extending number theory to walks on graphs},
  SIAM Journal on Discrete Mathematics 31~(2) (2017) 1428--1453.
\newblock \href {http://arxiv.org/abs/https://doi.org/10.1137/15M1054535}
  {\path{arXiv:https://doi.org/10.1137/15M1054535}}, \href
  {http://dx.doi.org/10.1137/15M1054535} {\path{doi:10.1137/15M1054535}}.
\newline\urlprefix\url{https://doi.org/10.1137/15M1054535}

\bibitem{Atkinson1999}
D.~Atkinson, F.~J. van Steenwijk,
  \href{https://doi.org/10.1119/1.19311}{Infinite resistive lattices}, American
  Journal of Physics 67~(6) (1999) 486--492.
\newblock \href {http://arxiv.org/abs/https://doi.org/10.1119/1.19311}
  {\path{arXiv:https://doi.org/10.1119/1.19311}}, \href
  {http://dx.doi.org/10.1119/1.19311} {\path{doi:10.1119/1.19311}}.
\newline\urlprefix\url{https://doi.org/10.1119/1.19311}

\bibitem{Cserti2011}
J.~Cserti, G.~Sz{\'{e}}chenyi, G.~D{\'{a}}vid,
  \href{https://doi.org/10.1088%2F1751-8113%2F44%2F21%2F215201}{Uniform tiling
  with electrical resistors}, Journal of Physics A: Mathematical and
  Theoretical 44~(21) (2011) 215201.
\newblock \href {http://dx.doi.org/10.1088/1751-8113/44/21/215201}
  {\path{doi:10.1088/1751-8113/44/21/215201}}.
\newline\urlprefix\url{https://doi.org/10.1088%2F1751-8113%2F44%2F21%2F215201}

\bibitem{Majumdar1991}
S.~N. Majumdar, D.~Dhar, {Height correlations in the Abelian sandpile model},
  Journal of Physics A: Mathematical and General 24~(7) (1991) L357--L362.
\newblock \href {http://dx.doi.org/10.1088/0305-4470/24/7/008}
  {\path{doi:10.1088/0305-4470/24/7/008}}.

\bibitem{Manna1992}
S.~S. Manna, D.~Dhar, S.~N. Majumdar,
  \href{https://link.aps.org/doi/10.1103/PhysRevA.46.R4471}{Spanning trees in
  two dimensions}, Physical Review A 46 (1992) R4471--R4474.
\newblock \href {http://dx.doi.org/10.1103/PhysRevA.46.R4471}
  {\path{doi:10.1103/PhysRevA.46.R4471}}.
\newline\urlprefix\url{https://link.aps.org/doi/10.1103/PhysRevA.46.R4471}

\bibitem{Viennot1989}
G.~X. Viennot, Heaps of pieces, {I}: Basic definitions and combinatorial
  lemmas, Annals of the New York Academy of Sciences 576~(1) (1989) 542--570.

\bibitem{flajolet2009analytic}
P.~Flajolet, R.~Sedgewick, Analytic combinatorics, Cambridge University Press,
  2009.

\bibitem{fill2005}
J.~A. Fill, P.~Flajolet, N.~Kapur, Singularity analysis, {{Hadamard}} products,
  and tree recurrences, Journal of Computational and Applied Mathematics
  174~(2) (2005) 271--313.
\newblock \href {http://dx.doi.org/10.1016/j.cam.2004.04.014}
  {\path{doi:10.1016/j.cam.2004.04.014}}.

\bibitem{Espinasse2019}
T.~{Espinasse}, P.~{Rochet}, {A coupling of the spectral measures at a vertex},
  The Electronic journal of Combinatorics 26.

\bibitem{Kenyon2000}
R.~Kenyon, \href{https://doi.org/10.1007/BF02392811}{{The asymptotic
  determinant of the discrete Laplacian}}, Acta Mathematica 185~(2) (2000)
  239--286.
\newblock \href {http://dx.doi.org/10.1007/BF02392811}
  {\path{doi:10.1007/BF02392811}}.
\newline\urlprefix\url{https://doi.org/10.1007/BF02392811}

\bibitem{bousquet-melou1993}
M.~{Bousquet-M{\'e}lou}, Q-{{{\'E}num{\'e}ration}} de {{Polyominos Convexes}},
  Journal of Combinatorial Theory, Series A 64~(2) (1993) 265--288.
\newblock \href {http://dx.doi.org/10.1016/0097-3165(93)90098-S}
  {\path{doi:10.1016/0097-3165(93)90098-S}}.

\bibitem{bousquet-melou1991}
M.~{Bousquet-M{\'e}lou}, {q-{\'E}num{\'e}ration de Polyominos Convexes}, Ph.D.
  thesis, Universit{\'e} du Qu{\'e}bec {\`a} Montr{\'e}al, {D{\'e}partement de
  math{\'e}matiques et d'informatique} (1991).

\bibitem{Benjamini2001}
I.~{Benjamini}, O.~{Schramm}, Recurrence of distributional limits of finite
  planar graphs, Electron. J. Probab. 6.

\end{thebibliography}

\end{document}